\documentclass[10pt]{amsart}
\usepackage{amssymb,url,rotating,color}
\usepackage[left=2cm,right=2cm,top=2cm,bottom=2.3cm]{geometry}
\newtheorem{theorem}{Theorem}
\newtheorem{thm}[theorem]{Theorem}

\newtheorem{conjecture}[theorem]{Conjecture}
\newtheorem{corollary}[theorem]{Corollary}

\newtheorem{lemma}[theorem]{Lemma}

\newtheorem{proposition}[theorem]{Proposition}
\newtheorem{remark}[theorem]{Remark}

\let\Im\relax
\DeclareMathOperator{\Im}{Im}
\let\Re\relax
\DeclareMathOperator{\Re}{Re}

\newcommand{\C}{\mathbb{C}}
\newcommand{\h}{\mathbb{H}}
\newcommand{\M}{\mathbb{M}}
\newcommand{\N}{\mathbb{N}}
\newcommand{\p}{\mathbb{P}}
\newcommand{\Q}{\mathbb{Q}}
\newcommand{\R}{\mathbb{R}}
\newcommand{\Z}{\mathbb{Z}}
\newcommand{\eps}{\varepsilon}
\DeclareMathOperator{\SL}{SL}
\DeclareMathOperator{\PSL}{PSL}
\DeclareMathOperator{\vol}{Vol}
\DeclareMathOperator{\lcm}{lcm}
\DeclareMathOperator{\Id}{Id}
\DeclareMathOperator{\spec}{Spec}
\newcommand{\E}{\mathcal{E}}
\newcommand{\rO}{\mathcal{O}}
\newcommand{\s}{\mathcal{S}}
\newcommand{\abs}[1]{\left\vert#1\right\vert}

\begin{document}

\title[Kronecker's limit formula, holomorphic modular functions and $q$-expansions]
{Kronecker's limit formula, holomorphic modular functions and $q$-expansions on certain moonshine
groups}
\author[J.~Jorgenson]{Jay Jorgenson}
\address{Department of Mathematics, The City College of New York, Convent Avenue at 138th Street,
New York, NY 10031 USA, e-mail: jjorgenson@mindspring.com}
\author[L.~Smajlovi\'c]{Lejla Smajlovi\'c}
\address{Department of Mathematics, University of Sarajevo, Zmaja od Bosne 35, 71\,000 Sarajevo,
Bosnia and Herzegovina, e-mail: lejlas@pmf.unsa.ba}
\author[H.~Then]{Holger Then}
\address{Department of Mathematics, University of Bristol, University Walk, Bristol, BS8 1TW,
United Kingdom, e-mail: holger.then@bristol.ac.uk}

\begin{abstract}
For any square-free integer $N$ such that the ``moonshine group'' $\Gamma_0(N)^+$ has genus zero,
the Monstrous Moonshine Conjectures relate the Hauptmoduli of $\Gamma_0(N)^+$ to certain
McKay-Thompson series associated to the representation theory of the Fischer-Griess monster
group.  In particular, the Hauptmoduli admits a $q$-expansion which has integer coefficients.
In this article, we study the holomorphic function theory associated to higher genus moonshine
groups $\Gamma_0(N)^+$.  For all moonshine groups of genus up to and including three, we prove that
the corresponding function field admits two generators whose $q$-expansions have integer
coefficients, has lead coefficient equal to one, and has minimal order of pole at infinity.
As corollary, we derive a polynomial relation which defines the underlying projective curve, and
we deduce whether $i\infty$ is a Weierstrass point.
Our method of proof is based on modular forms and includes extensive computer assistance, which,
at times, applied Gauss elimination to matrices with thousands of entries, each one of which was
a rational number whose numerator and denominator were thousands of digits in length.
\end{abstract}

\thanks{J.~J.\ acknowledges grant support from NSF and PSC-CUNY grants, and H.~T.\ acknowledges
support from EPSRC grant EP/H005188/1.}

\maketitle

\section{Introduction}

\subsection{Classical aspects of Klein's $j$-invariant}

The action of the discrete group $\PSL(2,\Z)$ on the hyperbolic upper half plane $\h$ yields a
quotient space $\PSL(2,\Z)\backslash\h$ which has genus zero and one cusp.  By identifying
$\PSL(2,\Z)\backslash\h$ with a well-known fundamental domain for the action of $\PSL(2,\Z)$ on
$\h$, we will, as is conventional, identify the cusp of $\PSL(2,\Z)\backslash\h$ with the point
denoted by $i\infty$.  From the uniformization theorem, we have the existence of a single-valued
meromorphic function on $\PSL(2,\Z)\backslash\h$ which has a simple pole at $i\infty$ and which
maps the one-point compactification of $\PSL(2,\Z)\backslash\h$ onto the one-dimensional
projective space $\p^1$.  Let $z$ denote the global coordinate on $\h$, and set $q=e^{2\pi i z}$,
which is a local coordinate in a neighborhood of $i\infty$ in the compactification of
$\PSL(2,\Z)\backslash\h$.  The bi-holomorphic map $f$ from the compactification of
$\PSL(2,\Z)\backslash\h$ onto $\p^1$ is uniquely determined by specifying constants
$c_{-1} \neq 0$ and $c_0$ such that the local expansion of $f$ near $i\infty$ is of the form
$f(q) = c_{-1}q^{-1} + c_0 + O(q)$ as $q\to0$.  For reasons coming from the theory of automorphic
forms, one chooses $c_{-1} = 1$ and $c_0 = 744$.  The unique function obtained by setting of
$c_{-1}=1$ and $c_0=744$ is known as Klein's $j$-invariant, which we denote by $j(z)$.

Let
$$
\Gamma_{\infty} = \left\{ \begin{pmatrix} 1& n\\ 0& 1 \end{pmatrix} \in \SL(2,\Z) \right\}.
$$
One defines the holomorphic Eisenstein series $E_k$ of even weight $k\geq4$ by the series
\begin{align}\label{E k defin}
E_k(z):= \sum_{\gamma \in \Gamma_\infty \backslash\PSL(2,\Z)} (cz+d)^{-k}
\quad \text{where} \quad
\gamma = \begin{pmatrix} * & * \\ c & d \end{pmatrix}.
\end{align}
Classically, it is known that
\begin{align}\label{j_via_Eisenstein}
j(z) = 1728E_4^3(z)/(E_4^3-E_6^2),
\end{align}
which yields an important explicit expression for the $j$-invariant.  As noted on page 90 of
\cite{Serre73}, which is a point we will emphasize in the next subsection, the $j$-invariant
admits the series expansion
\begin{align}\label{j_expansion}
j(z) = \frac{1}{q} + 744 + 196884 q + 21493760 q^2 + O(q^3)
\quad
\text{as $q \to 0$.}
\end{align}
An explicit evaluation of the coefficients in the expansion \eqref{j_expansion} was
established by Rademacher in \cite{Ra38}, and we refer the reader to the fascinating
article \cite{Kn90} for an excellent exposition on the history of the $j$-invariant
in the setting of automorphic forms.  More recently, there has been some attention
on the computational aspects of the Fourier coefficients in \eqref{j_expansion};
see, for example, \cite{BK03} and, for the sake of completeness, we mention the
monumental work in \cite{Edix11} with its far reaching vision.

For quite some time it has been known that the $j$-invariant has importance far beyond the
setting of automorphic forms and the uniformization theorem as applied to
$\PSL(2,\Z)\backslash\h$.  For example, in 1937 T.~Schneider initiated the study of the
transcendence properties of $j(z)$.  Specifically, Schneider proved that if $z$ is a quadratic
irrational number in the upper half plane then $j(z)$ is an algebraic integer, and if $z$ is an
algebraic number but not imaginary quadratic then $j(z)$ is transcendental.  More specifically,
if $z$ is any element of an imaginary quadratic extension of $\Q$, then $j(z)$ is an algebraic
integer, and the field extension $\Q[j(z),z]/\Q(z)$ is abelian.  In other words, special values
of the $j$-invariant provide the beginning of class field theory.  The seminal work of
Gross-Zagier \cite{GZ85} studies the factorization of the $j$-function evaluated at imaginary
quadratic integers, yielding numbers known as singular moduli, from which we have an abundance
of current research which reaches in various directions of algebraic and arithmetic number theory.

There is so much richness in the arithmetic properties of the $j$-invariant that we are unable to
provide an exhaustive list, but rather ask the reader to accept the above-mentioned examples as
indicating the important role played by Klein's $j$-invariant in number theory and algebraic
geometry.

\subsection{Monstrous moonshine}

Klein's $j$-invariant attained another realm of importance with the discovery of
``Monstrous Moonshine''.  We refer to the article \cite{Ga06a} for a fascinating survey of the
history of monstrous moonshine, as well as the monograph \cite{Ga06b} for a thorough account of
the underlying mathematics and physics surrounding the moonshine conjectures.  At this point,
we offer a cursory overview in order to provide motivation for this article.

In the mid-1900's, there was considerable research focused toward the completion of the list of
sporadic finite simple groups.  In 1973, B.~Fischer and R.~L.\ Griess discovered independently, but
simultaneously, certain evidence suggesting the existence of the largest of all sporadic simple
groups; the group itself was first constructed by R.~L.\ Griess in \cite{Gr82} and is now known as
``the monster'', or ``the friendly giant'', or ``the Fischer-Griess monster'', and we denote the
group by $\M$.  Prior to \cite{Gr82}, certain properties of $\M$ were deduced assuming its
existence, such as its order and various aspects of its character table.  At that time,
there were two very striking observations.  On one hand, A.~Ogg observed that the set of primes
which appear in the factorization of the order of $\M$ is the same set of primes such that the
discrete group $\Gamma_0(p)^+$ has genus zero.  On the other hand, J.~McKay observed that the
linear-term coefficient in \eqref{j_expansion} is the sum of the two smallest irreducible
character degrees of $\M$.  J.~Thompson further investigated McKay's observation in \cite{Tho79a}
and \cite{Tho79b}, which led to the conjectures asserting all coefficients in the expansion
\eqref{j_expansion} are related to the dimensions of the components of a graded module admitting
action by $\M$.  Building on this work, J.~Conway and S.~Norton established the
``monstrous moonshine'' conjectures in \cite{CN79} which more precisely formulated relations
between $\M$ and Klein's $j$-invariant \eqref{j_expansion}.  A graded representation for $\M$
was explicitly constructed by I.~Frenkel, J.~Lepowsky, and A.~Meurman in \cite{FLM88} thus
proving aspects of the McKay-Thompson conjectures from \cite{Tho79a} and \cite{Tho79b}.
Building on this work, R.~Borcherds proved a significant portion of the Conway-Norton
``monstrous moonshine'' conjectures in his celebrated work \cite{Bo92}.  More recently,
additional work by many authors (too numerous to list here) has extended ``moonshine'' to
other simple groups and other $j$-invariants associated to certain genus zero Fuchsian groups.

Still, there is a considerable amount yet to be understood within the framework of
``monstrous moonshine''.  Specifically, we call attention to the following statement
by T.~Gannon from \cite{Ga06a}:

\emph{In genus $> 0$, two functions are needed to generate the function field.  A complication
facing the development of a higher-genus Moonshine is that, unlike the situation in genus $0$
considered here, there is no canonical choice for these generators.}

In other words, one does not know the analogue of Klein's $j$-invariant for moonshine groups of
genus greater than zero from which one can begin the quest for ``higher genus moonshine''.

\subsection{Our main result}

The motivation behind this article is to address the above statement by T.~Gannon.  The methodology
we developed yields the following result, which is the main theorem of the present paper.

\begin{thm}\label{main theorem}
For any square-free $N$, let $\Gamma_0(N)^+$ denote the ``moonshine group'' which
is obtained by adding the Atkin-Lehner involutions to the congruence subgroup $\Gamma_0(N)$.
Then the function field associated to every genus one, genus two, and genus three moonshine group
admits two generators whose $q$-expansions have integer coefficients after the lead coefficient has
been normalized to equal one.  Moreover, the orders of poles of the generators at $i\infty$ are at
most $g+2$.
\end{thm}

In all cases, the generators we compute have the minimal poles possible, as can be shown by
the Weierstrass gap theorem.  Finally, as an indication of the explicit nature of our results,
we compute a polynomial relation associated to the underlying projective curve.
For all moonshine groups of genus two and genus three, we deduce whether $i\infty$
is a Weierstrass point.

In brief, our analysis involves four steps.  First, we establish an integrality theorem which
proves that if a holomorphic modular function $f$ on $\Gamma_0(N)^+$ admits a $q$-expansion of
the form $f(z) = q^{-a} + \sum_{k>-a}c_kq^k$, then there is an explicitly computable $\kappa$
such that if $c_k \in \Z$ for $k \leq \kappa$ then $c_k \in \Z$ for all $k$.  Second, we prove
the analogue of Kronecker's limit formula, resulting in the construction of a non-vanishing
holomorphic modular form $\Delta_N$ on $\Gamma_0(N)^+$; we refer to $\Delta_N$ as the
Kronecker limit function.  Third, we construct spaces of holomorphic modular functions on
$\Gamma_0(N)^+$ by taking ratios of holomorphic modular forms whose numerators are holomorphic
Eisenstein series on $\Gamma_0(N)^+$ and whose denominator is a power of the Kronecker limit
function.  Finally, we employ considerable computer assistance in order to implement an algorithm,
based on the Weierstrass gap theorem and Gauss elimination, to derive generators of the function
fields from our spaces of holomorphic functions.  By computing the $q$-expansions of the
generators out to order $q^{\kappa}$, the integrality theorem from the first step completes our
main theorem.

\subsection{Additional aspects of the main theorem}\label{additional aspects}

There are ways in which one can construct generators of the function fields associated to the
moonshine groups, two of which we now describe.  However, the methodology does not provide all
the information which we developed in the proof of our main theorem.

Classically, one can use Galois theory and elementary aspects of the moonshine group in order to
form modular functions using Klein's $j$-invariant.  In particular, the functions
$$
\sum_{v\vert N}j(vz) \quad \text{and} \quad \prod_{v\vert N}j(vz)
$$
generate all holomorphic functions which are $\Gamma_0(N)^+$ invariant and,
from (\ref{j_expansion}), admit $q$-expansions with lead coefficients equal to one and all
other coefficients are integers.
However, the order of poles at $i\infty$ are much larger than $g+2$.

From modern arithmetic algebraic geometry, we have another approach toward our main result.  We
shall first present the argument, which was first given to us by an anonymous reader of a previous
draft of the present article, and then discuss how our main theorem goes beyond the given argument.

The modular curve $X_0(N)$ exists over $\Z$, is nodal, and its cusps are disjoint $\Z$-valued
points.  For all primes $p$ dividing $N$, the Atkin-Lehner involutions $w_p$ of $X_0(N)$ commute
with each other and generate a group $G_N$ of order $2^r$.  Let $X_0(N)^+ = X_0(N)/G_N$, and let
$c$ denote the unique cusp of $X_0(N)^+$.  The cusp can be shown to be smooth over $\Z$; moreover,
all the fibers of $X_0(N)^+$ over $\spec(\Z)$ are geometrically irreducible.  Therefore, the
complement $U$ of $\spec(\Z)$ in $X_0(N)^+$ is affine.  The coordinate ring $\rO_N(U)$ is a
finitely generated $\Z$-algebra, with an increasing filtration by sub-$\Z$-modules $\rO_N(U)_k$
consisting of the functions with a pole of order at most $k$ along the cusp $c$.  The successive
quotients $\rO_N(U)_k / \rO_N(U)_{k-1}$ are free $\Z$-modules, each with rank either zero or one.
The $k$ for which the rank is zero form the gap-sequence of $X_0(N)^+_\Q$, the generic fiber of
$X_0(N)^+$.  If $k$ is not a gap, there is a non-zero element $f_k \in \rO_N(U)_k$, unique up to
sign and addition of elements from $\rO_N(U)_{k-1}$.  The $q$-parameter is obtained by considering
the formal parameter at $c$, which is the projective limit of the sequence of quotients
$(\rO_N/I^n)(X_0(N)^+)$ where $I$ is the ideal of $c$.  The projective limit is naturally
isomorphic to the formal power series ring $\Z[[q]]$.  With all this, if $f$ is any element of
$\rO_N(U_\Q)$, then there is an integer $n$ such that $nf$ has $q$-expansion in $\Z((q))$.  In
other words, the function field of $X_0(N)^+_\Q$ can be generated by two elements whose
$q$-expansions are in $\Z((q))$.

Returning to our main result, we can discuss the information contained in our main theorem which
goes beyond the above general argument.  The above approach from arithmetic algebraic geometry
yields generators of the function field with poles of small order with integral $q$-expansion;
however, it was necessary to ``clear denominators'' by multiplying through by the integer $n$,
thus allowing for the possibility that the lead coefficient is not equal to one.  In the case
that $X_0(N)^+$ has genus one, then results which are well-known to experts in the field
(though there is no specific reference, we refer to \cite{Con03} for an exposition) may be used
to show that the generators satisfy an integral Weierstrass equation; see, in particular,
Definition 2.4, Theorem 2.8 and the proof of Corollary 2.9.  However, these arguments do not
apply to the genus two and genus three cases which we exhaustively analyzed.  Furthermore, we
found that our result holds in all cases when $N$ is square-free, independent of the structure of
the gap sequence at $i\infty$; see, \cite{Kon03}.  As we found, for genus two $i\infty$ was not a
Weierstrass point for any $N$; however, for genus three, there were various types of gap sequences
for different levels $N$.  Further details of these findings will be mentioned in later sections
in the present article and described in detail in \cite{JST13c}.

Finally, given all of the above discussion, one naturally is led to the following conjecture.

\begin{conjecture}\label{Conjecture}
For any square-free $N$, the function field associated to any positive genus $g$ moonshine group
$\Gamma_0(N)^+$ admits two generators whose $q$-expansions have integer coefficients after the
lead coefficient has been normalized to equal one.  Moreover, the orders of poles of the
generators at $i\infty$ are at most $g+2$.
\end{conjecture}

\subsection{Our method of proof}

Let us now describe our theoretical results and computational investigations in greater detail.

For any square-free integer $N$, the subset of $\SL(2,\R)$ defined by
$$
\Gamma_0(N)^+=\left\{ e^{-1/2} \begin{pmatrix} a & b \\ c & d \end{pmatrix} \in \SL(2,\R):
\quad ad-bc=e, \quad a,b,c,d,e\in\Z, \quad e\mid N,\ e\mid a,\ e\mid d,\ N\mid c \right\}
$$
is an arithmetic subgroup of $\SL(2,\R)$, which, as discussed in \cite{JST12},
we call a moonshine group.
As shown in \cite{Cum04}, there are precisely $44$ such moonshine groups which have genus zero
and $38$, $39$, and $31$ which have genus one, two, and three, respectively.  These groups of
genus up to three will form a considerable portion of the focus in this article.
There are two reasons for focusing on the groups $\Gamma_0(N)^+$ for square-free $N$.
First, for any positive integer $N$, the groups $\Gamma_0(N)^+$ are of moonshine-type
(see, e.g., Definition 1 from \cite{Ga06a}).  Second, should $N$ not be square-free, then there
exist genus zero groups $\Gamma_0(N)^+$, namely when $N=25$, $49$ and $50$, but those groups
correspond to "ghost" classes of the monster.  In summary, we are paying attention to the
known results from ``monstrous moonshine''.

For arbitrary square-free $N$, the discrete group $\Gamma_0(N)^+$ has one-cusp,
which we denote by $i\infty$.  Associated to the cusp of $\Gamma_0(N)^+$ one
has, from spectral theory and harmonic analysis, a well-defined non-holomorphic Eisenstein series
denoted by $\E_\infty(z,s)$.  The real analytic Eisenstein series
$\E_\infty(z,s)$ is defined for $z \in \h$ and $\Re(s)>1$ by
\begin{align}\label{parabol Eis defin}
\E_\infty(z,s) = \sum_{\gamma\in\Gamma_\infty(N)\backslash\Gamma_0(N)^+}\Im(\gamma z)^s,
\end{align}
where $\Gamma_\infty(N)$ is the parabolic subgroup of $\Gamma_0(N)^+$.  Our first
result is to determine the Fourier coefficients of $\E_\infty(z,s)$ in terms of
elementary arithmetic functions, from which one obtains the meromorphic continuation
of the real analytic Eisenstein series $\E_\infty(z,s)$ to all $s \in \C$.

As a corollary of these computations, it is immediate that $\E_\infty(z,s)$ has no pole in
the interval $(1/2,1)$.  Consequently, we prove that groups $\Gamma_0(N)^+$ have no residual
spectrum besides the obvious one at zero.

Using our explicit formulas for the Fourier coefficients of $\E_\infty(z,s)$,
we are able to study the special values at $s=1$ and $s=0$, which, of course, are related by
the functional equation for $\E_\infty(z,s)$.  As a result, we arrive at the
following generalization of Kronecker's limit formula.
For any square-free $N$ which has $r$ prime factors, the real analytic Eisenstein series
$\E_\infty(z,s)$ admits a Taylor series expansion of the form
$$
\E_\infty(z,s) = 1+ (\log\left(\sqrt[2^r]{\prod_{v \mid N} \abs{ \eta(vz)}^4}
\cdot \Im(z) \right))\cdot s + O(s^2), \text{ as } s\to0,
$$
where $\eta(z)$ is Dedekind's eta function associated to $\PSL(2,\Z)$.  From the
modularity of $\E_\infty(z,s)$, one concludes that
$\prod_{v \mid N} \abs{ \eta(vz)}^4 (\Im(z))^{2^r}$ is invariant with respect to the
action by $\Gamma_0(N)^+$.  Consequently, there exists a multiplicative
character $\eps_N(\gamma)$ on $\Gamma_0(N)^+$ such that one has the identity
\begin{align}\label{Dedekind_sums}
\prod_{v \mid N} \eta(v \gamma z) =\eps_N(\gamma)(cz+d)^{2^{r-1}}\prod_{v \mid N} \eta(v z)
\quad \text{for all $\gamma  = \begin{pmatrix} * & * \\ c & d \end{pmatrix} \in \Gamma_0(N)^+$.}
\end{align}
We study the order of the character, and we define $\ell_N$ to be the integer so that
$\eps_N^{\ell_N} = 1$.  A priori, it is not immediate that $\ell_N$ is finite for general $N$.
Classically, when $N=1$, the study of the character $\eps_1$ on $\PSL(2,\Z)$ is the beginning
of the theory of Dedekind sums; see \cite{Lang76}.  For general $N$, we prove that $\ell_N$ is
finite and, furthermore, can be evaluated by the expression
$$
\ell_N = 2^{1-r}\lcm\Big(4,\ 2^{r-1}\frac{24}{(24,\sigma(N))}\Big)
$$
where $\lcm(\cdot,\cdot)$ denotes the least common multiple function and $\sigma(N)$ stands for
the sum of divisors of $N$.

With the above notation, we define the Kronecker limit function $\Delta_N(z)$ associated to
$\Gamma_0(N)^+$ to be
\begin{align}\label{Kronecker_limit}
\Delta_N(z) = \left(\prod_{v \mid N} \eta(v z)\right)^{\ell_N},
\end{align}
and we let $k_N$ denote the weight of $\Delta_N$.  By combining \eqref{Dedekind_sums}
and \eqref{Kronecker_limit}, we get that $k_N=2^{r-1}\ell_N$.
In summary, $\Delta_N(z)$ is a non-vanishing, weight $k_N$
holomorphic modular form with respect to $\Gamma_0(N)^+$.

Analogous to the setting of $\PSL(2,\Z)$, one defines the holomorphic Eisenstein series
of even weight $k\geq 4$ associated to $\Gamma_0(N)^+$ by the series
\begin{align}\label{E k, p defin}
E_k^{(N)}(z)= \sum_{\gamma \in \Gamma_\infty(N) \backslash\Gamma_0(N)^+} (cz+d)^{-k}
\quad \text{where} \quad
\gamma = \begin{pmatrix} * & * \\ c & d \end{pmatrix}.
\end{align}
We show that one can express $E_{2m}^{(N)}$ in terms of $E_{2m}$, the holomorphic Eisenstein
series associated to $\PSL(2,\Z)$; namely, for $m\geq 2$ one has the relation
\begin{align}\label{holo_eisen_N}
E_{2m}^{(N)}(z)= \frac1{\sigma_m(N)} \sum_{v \mid N}v^m E_{2m}(vz).
\end{align}

With the above analysis, we are now able to construct holomorphic modular functions on the
space $\Gamma_0(N)^+\backslash\h$.  For any non-negative integer $M$, the function
\begin{align}\label{rational_function_b}
F_b(z) := \prod_\nu\left(E_{m_\nu}^{(N)}(z)\right)^{b_\nu}\Big/\big(\Delta_N(z)\big)^M
\quad \text{where} \quad
\sum_\nu b_\nu m_\nu = Mk_N \quad \text{and} \quad b = (b_1, \ldots)
\end{align}
is a holomorphic modular function on $\Gamma_0(N)^+\backslash\h$, meaning a weight zero modular
form with polynomial growth near $i\infty$.  The vector $b=(b_\nu)$ can be viewed as a weighted
partition of the integer $k_NM$ with weights $m=(m_\nu)$ formed by the weights, in the
sense of modular forms, of the holomorphic Eisenstein series.  For considerations to be
described below, we let $\s_M$ denote the set of all possible rational functions
defined in \eqref{rational_function_b} by varying the vectors $b$ and $m$ yet keeping $M$ fixed.
Trivially, $\s_0$ consists of the constant function $\{1\}$, which is convenient to include in
our computations.

Dedekind's eta function can be expressed by the well-known product formula, namely
$$
\eta^{24}(z)=q\prod_{n=1}^\infty(1-q^n)^{24}
\quad \text{where} \quad
q=e^{2\pi i z} \quad \text{with} \quad z\in\h,
$$
and the holomorphic Eisenstein series associated to $\PSL(2,\Z)$ admits the
$q$-expansion
\begin{align}\label{holo_eisen_one}
E_k(z)=1-\frac{2k}{B_k}\sum_{\nu=1}^\infty \sigma_{k-1}(\nu)q^\nu, \quad
\text{as $q \to 0$.}
\end{align}
where $B_k$ denotes the $k$-th Bernoulli number and $\sigma$ is the generalized
divisor function
$$
\sigma_{\alpha}(m)= \sum_{\delta \mid m} \delta^{\alpha}.
$$

It is immediate that each function $F_b$ in \eqref{rational_function_b} admits a $q$-expansion
with rational coefficients.  However, it is clear that such coefficients are certainly not
integers and, actually, can have very large numerators and denominators.  Indeed, when
combining \eqref{holo_eisen_N} and \eqref{holo_eisen_one}, one gets that
$$
E_{2m}^{(N)}(z) = \frac1{\sigma_m(N)} \sum_{v \mid N}v^m E_{2m}(vz)
= 1 - \frac{4m}{B_{2m}\sigma_m(N)}q + O(q^2), \quad
\text{as $q \to 0$.}
$$
which, evidentially, can have a large denominator when $m$ and $N$ are large.  In addition,
note that the function in \eqref{rational_function_b} is defined with a product of holomorphic
Eisenstein series in the numerators, so the rational coefficients in the $q$-expansion
of \eqref{rational_function_b} are even farther removed from being easily described.

With the above theoretical background material, we implement the algorithm described in
section~\ref{Algorithm} to find a set of generators for the function field associated to the
genus $g$ moonshine group $\Gamma_0(N)^+$.  In addition to the theoretical results outlined
above, we obtain the following results which are based on our computational investigations.
\begin{enumerate}
\item For all genus zero moonshine groups, the algorithm concludes successfully, thus yielding
the $q$-expansion of the Hauptmoduli.  In all such cases, the $q$-expansions have positive
integer coefficients as far as computed.  The computations were completed for all $44$ different
genus zero moonshine groups $\Gamma_0(N)^+$.
\item For all genus one moonshine groups, the algorithm concludes successfully, thus yielding
the $q$-expansions for two generators of the associated function fields.  In all such cases,
each $q$-expansion has integer coefficients as far as computed.  The computations were
completed for all $38$ different genus one moonshine groups $\Gamma_0(N)^+$.  The function whose
$q$-expansion begins with $q^{-2}$ is the Weierstrass $\wp$-function in the coordinate $z\in\h$
for the underlying elliptic curve.
\item For all $38$ different genus one moonshine groups, we computed a cubic relation satisfied
by the two generators of the function fields.
\item In addition, we consider all moonshine groups of genus two and three.  In every
instance, the algorithm concludes successfully, yielding generators for the function fields whose
$q$-expansions admit integer coefficients as far as computed.  Only for the function field
generators on $X_{510}$ an additional base change becomes necessary.
\item We extend all $q$-expansions out to order $q^\kappa$, where $\kappa$ is given in
Tables~\ref{tab:kappa genus zero} and~\ref{tab:kappa genus one} or evaluated according to
Remark~\ref{rem:kappa algorithm}, thus showing that the field generators have integer
$q$-expansions.
\end{enumerate}

The fact that all of the $q$-expansions which we uncovered have integer coefficients is
not at all obvious and leads us to believe there is deeper, so-far hidden, arithmetic structure
which perhaps can be described as ``higher genus moonshine''.

In some instances, the computations from our algorithm were elementary
and could have been completed without computer assistance.  For instance,
when $N=5$ or $N=6$, the first iteration of the algorithm used a set with
only two functions to conclude successfully.  However, as $N$ grew, the
complexity of the computations became quite large.  As an example, for $N=71$, which is genus
zero and appears in ``monstrous moonshine'', the smallest non-zero weight for a denominator
in \eqref{rational_function_b} was $4$, but we needed to consider all functions whose
numerators had weight up to $40$, resulting in $362$ functions whose largest pole had order
$120$.  The most computationally extensive genus one example was $N=79$ where the smallest
non-zero weight denominator in \eqref{rational_function_b} was 12, but we needed to consider
all functions whose numerator had weight up to $84$, resulting in $13,158$ functions whose
largest pole had order $280$.

As one can imagine, the data associated to the $q$-expansions we considered is massive.  In
some instances, we encountered rational numbers whose numerators and denominators each occupied
a whole printed page.  In addition, in the cases where the algorithm required
several iterations, the input data of $q$-expansions of all functions were stored in computer
files which if printed would occupy hundreds of thousands of pages.  As an example of the size
of the problem we considered, it was necessary to write computer programs to search the output
from the Gauss elimination to determine if all coefficients of all $q$-expansions were integers
since the output itself, if printed, would occupy thousands of pages.

The computational results are summarized in sections~\ref{genus zero},~\ref{genus one}
and~\ref{higher genus}.  In Table~\ref{tab:y & x genus one}, Table~\ref{tab:y & x genus two}, and
Table~\ref{tab:y & x genus three}, we list the $q$-expansions of the two generators of function
fields associated to each genus one, genus two, and genus three moonshine group.  As T.~Gannon's
comment suggests, the information summarized in those tables does not exist elsewhere.
In Table~\ref{pol:genus one}, Table~\ref{pol:genus two}, and Table~\ref{pol:genus three}, we
list the polynomial relations satisfied by the generators of Table~\ref{tab:y & x genus one},
Table~\ref{tab:y & x genus two}, and Table~\ref{tab:y & x genus three}, respectively.  The
stated results, in particular the $q$-expansions, were limited solely by space considerations;
a thorough documentation of our findings will be given in forthcoming articles.

We will supplement the dissemination of our study as summarized in the present article
by making available the computer programs as well as input
and output data supporting the statements we present here.  In other words:

\emph{We will make available all input and output information and computer programs
associated to the computational investigations undertaken in the present article.}

\subsection{Further studies}

As stated above, Klein's $j$-invariant can be written in terms of holomorphic Eisenstein series
associated to $\PSL(2,\Z)$.  By storing all information from the computations from Gauss
elimination, we obtain similar expressions for the Hauptmoduli for all genus zero
moonshine groups.  As one would imagine based upon the above discussion, some of the
expressions will be rather large.  In \cite{JST13a},
we will report of this investigation, which in many instances yields new relations for the
$j$-invariants and for holomorphic Eisenstein series themselves.  For the genus one groups,
the computations from Gauss elimination produce expressions for the Weierstrass $\wp$-function,
in the coordinate on $\h$, in terms of holomorphic Eisenstein series; these computations will
be presented in \cite{JST13b}.  Finally, in \cite{JST13c}, we will study higher genus examples,
and in some instances compute the projective equation of the underlying algebraic curve.

\subsection{Outline of the paper}

In section~\ref{background material}, we will establish notation and recall various background
material.  In section~\ref{integrality}, we prove that if a certain number of coefficients in the
$q$-expansions of the generators are integers, then all coefficients are integers.  In
section~\ref{parabolic Eisenstein}, we will compute the Fourier expansion of the non-holomorphic
Eisenstein series associated to $\Gamma_0(N)^+$.  The generalization of Kronecker's limit formula
for moonshine groups is proven in section~\ref{Kronecker limit}, including an investigation of its
weight and the order of the associated Dedekind sums.  In section~\ref{holomorphic Eisenstein}, we
relate the holomorphic Eisenstein series associated to $\Gamma_0(N)^+$ to holomorphic Eisenstein
series on $\PSL(2,\Z)$, as cited above.  Further details regarding the algorithm we develop and
implement are given in section~\ref{Generators of function fields}.  Results of our computational
investigations are given in section~\ref{genus zero} for genus zero, section~\ref{genus one}
for genus one, and section~\ref{higher genus} for genus two and genus three.  Finally, in
section~\ref{concluding remark}, we offer concluding remarks and discuss directions for future
study, most notably our forthcoming articles \cite{JST13a}, \cite{JST13b}, and \cite{JST13c}.

\subsection{Closing comment}

The quote we presented above from \cite{Ga06a} indicates that ``higher genus moonshine'' has
yet to have the input from which one can search for the type of mathematical clues that are found
in McKay's observation involving the coefficients of Klein's $j$-invariant or in Ogg's computation
of the levels of all genus zero moonshine groups and their appearance in the prime factorization
of the order of the Fischer-Griess monster $\M$.  It is our hope that someone will recognize
some patterns in the $q$-expansions we present in this article, as well as in \cite{JST13a},
\cite{JST13b}, and \cite{JST13c}, and then, perhaps, higher genus moonshine will manifest itself.

\section{Background material}\label{background material}

\subsection{Preliminary notation}

Throughout we will employ the standard notation for several arithmetic quantities and functions,
including:  The generalized divisor function $\sigma_a$, Bernoulli numbers $B_k$, the M\"obius
function $\mu$, and the Euler totient function $\varphi$, the Jacobi symbol
$\left(\frac{p}{d}\right)$, the greatest common divisor function $(\cdot,\cdot)$, and the least
common multiple function $\lcm(\cdot,\cdot)$.  Throughout the paper we denote by
$\{p_i\}$, $i=1,\ldots,r$, a set of distinct primes and by $N=p_1\cdots p_r$ a square-free,
positive integer.

The convention we employ for the Bernoulli numbers follows \cite{Zag08} which is slightly
different than \cite{Serre73} although, of course, numerical evaluations agree when following
either set of notation.  For a precise discussion, we refer to the footnote on page 90 of
\cite{Serre73}.

\subsection{Certain moonshine groups}

As stated above, $\h$ denotes the hyperbolic upper half plane with global variable $z \in \C$
with $z = x + iy$ and $y > 0$, and we set $q = e^{2\pi i z}=e(z)$.

The subset of $\SL(2,\R)$, defined by
\begin{align}\label{defn moons group}
\Gamma_0(N)^+=\left\{ e^{-1/2} \begin{pmatrix} a & b \\ c & d \end{pmatrix} \in \SL(2,\R):
\quad ad-bc=e, \quad a,b,c,d,e\in\Z, \quad e\mid N,\ e\mid a,\ e\mid d,\ N\mid c \right\}
\end{align}
is an arithmetic subgroup of $\SL(2,\R)$, called moonshine group.  We denote by
$\overline{\Gamma_0(N)^+} = \Gamma_0(N)^+ / \{\pm\Id\}$ the corresponding subgroup of $\PSL(2,\R)$.

Basic properties of $\Gamma_0(N)^+$, for square-free $N$ are derived in \cite{JST12} and
references therein.  In particular, we use that
the surface $X_N=\overline{\Gamma_0(N)^+}\backslash\h$ has exactly one cusp, which can
be taken to be at $i\infty$.  Let $g_N$ denote the genus of $X_N$.

\subsection{Function fields and modular forms}

The set of meromorphic functions on $X_N$ is a function field, meaning a degree one
transcendental extension of $\C$.  If $g_N=0$, then the function
field of $X_N$ is isomorphic to $\C(j)$ where $j$ is a single-valued meromorphic function
on $X_N$.  If $g_N>0$, then the function field of $X_N$
is generated by two elements which satisfy a polynomial relation.

A meromorphic function $f$ on $\h$ is a weight $2k$ meromorphic modular form if
we have the relation
$$
f(z) = (cz+d)^{-2k}f\left(\frac{az+b}{cz+d}\right)
\quad \text{for any} \quad
\gamma = \begin{pmatrix} a & b \\ c & d \end{pmatrix} \in \Gamma_0(N)^+.
$$
In other language, the differential $f(z)(\text{d}z)^k$ is $\Gamma_0(N)^+$ invariant.

The product, resp.\ quotient, of weight $k_1$ and weight $k_2$ meromorphic forms
is a weight $k_1+k_2$, resp.\ $k_1-k_2$, meromorphic form.  Since
$\big(\begin{smallmatrix}1 & 1 \\ 0 & 1 \end{smallmatrix}\big) \in \Gamma_0(N)^+$, each
holomorphic modular form admits a Laurent series expansion which, when setting
$q = e^{2\pi i z}$, can be written as
$$
f(z) = \sum_{m=-\infty}^{\infty} c_m q^m.
$$
As is common in the mathematical literature, we consider forms such that $c_m=0$ whenever
$m < m_0$ for some $m_0\in \Z$.  If $c_{m_0} \neq 0$, the function $f$ frequently will be
re-scaled so that $c_{m_0}=1$ when $c_m = 0$ for $m < m_0$; the constant $m_0$ is the
order of the pole of $f$ at $i\infty$.

If $X_N$ has genus zero, then by the Hauptmoduli $j_N$ for $X_N$ one means
the weight zero holomorphic form which is the generator of the function field on $X_N$.
If $g_N>0$, then we will use the notation $j_{1;N}$ and $j_{2;N}$ to denote
two generators of the function field.

Unfortunately, the notation of hyperbolic geometry writes the local coordinate on $\h$ as $x+iy$,
and the notation of the algebraic geometry of curves uses $x$ and $y$ to denote the generators
of the function field under consideration; see, for example, page 31 of \cite{Lang82}.  We will
follow both conventions and provide ample discussion in order to prevent confusion.

\section{Integrality of the coefficients in the $q$-expansion}\label{integrality}

In this section we prove that integrality of all coefficients in the $q$-expansion of the
Hauptmodul $j_N$ when $g_N=0$ and the generators $j_{1;N}$ and $j_{2;N}$ when $g_N>0$
can be deduced from integrality of a certain finite number of coefficients.  Also, our proof
yields an effective bound on the number of coefficients needed to test for
integrality.  First, we present a proof in the case
when genus is zero, followed by a proof for the higher genus setting which is a
generalization of the genus zero case.  The proof is based on the property of Hecke operators
and Atkin-Lehner involutions.  We begin with a simple lemma.

\begin{lemma}
For any prime $p$ which is relatively prime to $N$, let $T_p$ denote the unscaled Hecke
operator which acts on $\Gamma_0(N)$ invariant functions $f$,
$$
T_p(f)(z) = f(pz) + \sum_{b=0}^{p-1}f\left(\frac{z+b}{p}\right).
$$
If $f$ is a holomorphic modular form on $\Gamma_0(N)^+$, then $T_p(f)$ is also a
holomorphic modular form on $\Gamma_0(N)^+$.
\end{lemma}
\begin{proof}
The form $f$ is $\Gamma_0(N)^+$ invariant, hence $f$ is $\Gamma_0(N)$ invariant and
so is $T_p(f)$.  By Lemma 11 of \cite{AtLeh70}, if $W$ is any
coset representative of the quotient group $\Gamma_0(N)\backslash \Gamma_0(N)^+$, then
$T_p(f)$ is also $W$ invariant since $f$ is $W$ invariant.  Therefore, $T_p(f)$ is a
$\Gamma_0(N)^+$ invariant holomorphic form.
\end{proof}

\begin{thm}\label{genus_zero_integrality}
Let $j_N$ be the Hauptmodul for a genus zero moonshine group $\Gamma_0(N)^+$.  Let $p_1$
and $p_2$ with $p_2 > p_1$ be distinct primes which are relatively prime to $N$.  If the
$q$-expansion of $j_N$ contains integer coefficients out to order $q^\kappa$ for some
$\kappa \geq (p_2/(p_2-1))\cdot (p_1p_2)^2$, then all further coefficients in the
$q$-expansion of $j_N$ are integers.
\end{thm}

\begin{proof}
Let $c_k$ denote the coefficient of $q^k$ in the $q$-expansion of $j_N$.  Let $m > \kappa$ be
arbitrary integer.  Assume that $c_k$ is an integer for $k < m$, and let us now prove that
$c_m$ is an integer.  To do so, let us consider the $\Gamma_0(N)^+$ invariant functions
$$
T_{p_1}(j_N^{\ r_1})(z) \quad \text{and} \quad T_{p_2}(j_N^{\ r_2})(z)
$$
for specific integers $r_1$ and $r_2$ depending on $m$ which are chosen as follows.
First, take $r_1$ to be the unique integer such that
$$
r_1 \equiv (m+1) \mod p_1 \quad \text{and} \quad 1 \leq r_1 \leq p_1.
$$
We now wish to choose $r_2$ to be the smallest positive integer such that
$$
r_2 \equiv (m+1) \mod p_2 \quad \text{with} \quad r_2 \geq 1 \quad \text{and} \quad
(r_1p_1,r_2p_2) = 1.
$$
Let us argue in general the existence of $r_2$ and establish a bound in terms of
$p_1$ and $p_2$, noting that for any specific example one may be able to choose
$r_2$ much smaller than determined by the bound below.

Since $p_1$ and $p_2$ are primes and $1 \leq r_1 \leq p_1 < p_2$, we have
that $(p_2,r_1p_1) = 1$.  By applying the Euclidean algorithm and Bezout's identity,
there exists an $r\geq 1$ such that $rp_2 \equiv 1 \mod r_1p_1$.  Choose $A$ be the
smallest positive integer such that $A \equiv r^2 - (m+1)r \mod r_1p_1$ and
$$
r_2 = m+1 + Ap_2 \quad \text{and} \quad 1 \leq r_2 \leq r_1p_1p_2.
$$
Clearly, $r_2 \equiv (m+1)\mod p_2$ and, furthermore
\begin{align*}
r_2p_2 & \equiv (m+1)p_2 + Ap_2^2 \mod r_1p_1 \\
& \equiv (m+1)p_2 + (r^2 - (m+1)r)p_2^2 \mod r_1p_1 \\
& \equiv 1 \mod r_1p_1.
\end{align*}
In particular, $r_1p_1$ and $r_2p_2$ are relatively prime,
and we have the bounds $r_1p_1\leq p_1^2$ and $r_2p_2 \leq p_1^2p_2^2$.

With the above choices of $r_1$ and $r_2$, set $f_1=j_N^{\ r_1}$ and
$f_2=j_N^{\ r_2}$.  The function $f_1$ is a $\Gamma_0(N)^+$ invariant modular
function with pole of order $r_1$ at $i\infty$.  Therefore, there is a
polynomial $P_{r_1,p_1}(x)$ of degree $r_1p_1$ such that
\begin{align}\label{Hecke_equation_zero}
T_{p_1}(f_1) = P_{r_1,p_1}(j_N).
\end{align}
If we write the $q$-expansion of $f_1$ as
$$
f_1(z) = \sum_{k \geq -r_1}b_kq^k,
$$
then
$$
T_{p_1}(f_1)(z) = \sum_{k \geq -r_1}b_kq^{kp_1} +
\sum_{\substack{k \geq -r_1 \\ k \equiv 0 \mod p_1}}p_1b_kq^{k/p_1}.
$$
The coefficients $b_k$ are determined by the binomial theorem and the coefficients
of $j_N$, namely by
$$
f_1(z) = (j_N(z))^{r_1} = \left(q^{-1} + c_1q^1 + \ldots + c_kq^k + \ldots \right)^{r_1}.
$$
By assumption, $c_1,\ldots, c_{m-1}$ are integers, and
$m-1 \geq (p_2/(p_2-1))(p_1p_2)^2 > r_1$.
From this, we conclude that $b_{-r_1}, \ldots, b_{m-r_1}$ are integers.  In particular,
all coefficients of $T_{p_1}(f_1)(z)$ out to order $q^0$ are integers, from which we can
compute the coefficients of $P_{r_1,p_1}$ and conclude that the polynomial $P_{r_1,p_1}$
has integer coefficients, in particular the lead coefficient is one.

Let us determine the first appearance of the coefficient $c_m$ in $T_{p_1}(f_1)$.  First,
the smallest $k$ where $c_m$ appears in the formula for $b_k$ is when $k=m+1-r_1$, and
then we have that
$$
b_{m+1-r_1}
= r_1c_m + \text{an integer determined by binomial coefficients and $c_k$ for $k < m$.}
$$
By our choice of $r_1$, the index
$m+1-r_1$ is positive and divisible by $p_1$.  As a result, the first appearance
of $c_m$ in the expansion of $T_{p_1}(f_1)$ is within the coefficient of $q^d$ where
$d = (m+1-r_1)/p_1$.  Going further, we have that the coefficient of $q^d$ for
$d = (m+1-r_1)/p_1$ in the expansion of $T_{p_1}(f_1)$ is of the form $p_1r_1c_m$
plus an integer which can be determined by binomial coefficients and $c_k$ for $k < m$.

Let us now determine the first appearance of $c_m$ in $P_{r_1,p_1}(j_N)$.  Again,
by the binomial theorem, we have that $c_m$ first appears as a coefficient of $q^e$ where
$e = m-r_1p_1+1$.  By the choice of $m$ and $r_1$, since $p_1<p_2$, one has that
$$
m> p_1p_2 > r_1(p_1 +1)-1 = \frac{r_1p_1^2 - r_1-p_1+1}{p_1-1}.
$$
The above inequality is equivalent to
$$
\frac{m+1-r_1}{p_1}<m-r_1p_1 +1;
$$
in other words, $d<e$.  As a consequence, we have that the coefficient of $q^d$ in
$P_{r_1,p_1}(j_N)$ can be written as a polynomial expression involving binomial
coefficients and $c_k$ for $k < m$.  By induction, the coefficient of $q^d$ in
$P_{r_1,p_1}(j_N)$ is an integer.

In summary, by equating the coefficients of $q^d$ where $d=(m+1-r_1)/p_1$ in the formula
\eqref{Hecke_equation_zero}, on the left-hand-side we get an expression of the formula
$r_1p_1c_m$ plus an integer, and on the right-hand-side we get an integer.  Therefore,
$c_m$ is a rational number whose denominator is a divisor of $r_1p_1$.

Let us now consider $T_{p_2}(j_N^{\ r_2})$.  Since $m+1-r_2$ is divisible
by $p_2$, we consider the coefficient of $q^{d'}$ where $d' = (m+1-r_2)/p_2$.  By the
choice of $m$ and $r_2$, recalling that $r_2 \leq p_1^2 p_2$, we have
$$
m \geq \frac{p_2 ^2}{p_2-1} p_1^2 p_2 > \frac{p_2^2-1}{p_2-1}r_2 > r_2(p_2+1)-1
$$
and this is equivalent to
$$
\frac{m+1-r_2}{p_2}<m-r_2p_2 +1,
$$
or $d'<e'=m-r_2p_2 +1$.  With all this, we conclude, analogously as in the first case that
$c_m$ is a rational number whose denominator is a divisor of $r_2p_2$.
However, $r_1p_1$ and $r_2p_2$ are relatively prime, hence $c_m$ is an integer.

By induction on $m$, the proof of the theorem is complete.
\end{proof}

\begin{thm}\label{genus_notzero_integrality}
Let $\Gamma_0(N)^+$ have genus greater than zero, and assume there exists two holomorphic
modular functions $j_{1;N}$ and $j_{2;N}$ which generate the function field associated to
$\Gamma_0(N)^+$.  Furthermore, assume that the $q$-expansions of $j_{1;N}$ and $j_{2;N}$ are
normalized in the form
$$
j_{1;N}(z) = q^{-a_1} + O\left(q^{-a_1+1}\right) \quad \text{and} \quad
j_{2;N}(z) = q^{-a_2} + O\left(q^{-a_2+1}\right) \quad \text{with} \quad a_1 \leq a_2.
$$
Let $p_1$ and $p_2$ with $p_2 > p_1$ be distinct primes which are relatively prime to
$a_1a_2N$.  Assume the $q$-expansions of $j_{1;N}$ and $j_{2;N}$ contain integer coefficients out
to order $q^\kappa$ with $\kappa \geq a_2(p_2/(p_2-1))\cdot (p_1p_2)^2$.  Then all further
coefficients in the $q$-expansions of $j_{1;N}$ and $j_{2;N}$ are integers.
\end{thm}

\begin{proof}
The argument is very similar to the proof of Theorem~\ref{genus_zero_integrality}.  Let $c_{l;k}$
denote the coefficient of $q^k$ in the $q$-expansion of $j_{l;N}$, $l\in \{1,2\}$.  Assume that
$c_{l;k}$ is an integer for $l\in \{1,2\}$ and $k < m$, and let us prove that $c_{1;m}$ and
$c_{2;m}$ are integers.

For $i,l \in \{1,2\}$, we study the expression
\begin{align}\label{T_p in two variables}
T_{p_i}(j_{l;N}^{\ r_i})(z) = (j_{l;N}(z))^{r_i p_i} + Q_{i,l}(j_{1;N},j_{2;N}).
\end{align}
for certain polynomials $Q_{i,l}$ of two variables.  The left-hand-side of the above equation has
a pole at $i\infty$ of order $a_lp_ir_i$.  Hence
$$
Q_{i,l}(x,y)=\sum_{0\leq a_1n_1+a_2n_2 < a_lp_ir_i} b_{i,l;n_1,n_2}x^{n_1}y^{n_2}
\quad \text{for $i,l\in\{1,2\}$}.
$$
The existence of $Q_{i,l}(x,y)$ follows from the assumption that $j_{1;N}$ and $j_{2;N}$ generate
the function field associated to $\Gamma_0(N)^+$ and the observation that
$T_{p_i}(j_{l;N}^{\ r_i}) - (j_{l;N})^{r_ip_i}$ is $\Gamma_0(N)^+$ invariant and has a pole of order
less than $a_lr_ip_i$.

Of course, the polynomials $Q_{i,l}$ are not unique since $j_{1;N}$ and $j_{2;N}$ satisfy a
polynomial relation.  This does not matter.  We introduce the following canonical choice in order
to uniquely determine the polynomials.  Consider the coefficient $b_{i,l;n_1,n_2}$.  If there exist
non-negative integers $n'_1$ and $n'_2$ such that $a_1n'_1+a_2n'_2 = a_1n_1+a_2n_2$ with
$n'_1 < n_1$ then we set $b_{i,l;n_1,n_2}$ equal to zero.

Integrality of coefficients of $j_{1;N}$ and $j_{2;N}$ out to order $q^{\kappa}$ with
$\kappa \geq a_2(p_2/(p_2-1))\cdot (p_1p_2)^2$ implies that all coefficients of
$T_{p_i}(j_{l;N}^{\ r_i})(z)$ out to order $q^0$ are integers.  The coefficient $b_{i,l;n_1,n_2}$
of the polynomial $Q_{i,l}(x,y)$ first appears on the right-hand-side of
\eqref{T_p in two variables} as a coefficient multiplying $q^{-(n_1a_1+n_2a_2)}$.  The canonical
choice of coefficients enables us to deduce, inductively in the degree ranging from $-a_lr_ip_i+1$
to zero that all coefficients $b_{i,l;n_1,n_2}$ are integers.

Having established that the coefficients of the polynomials $Q_{i,l}$ are integers, now we wish
to determine two values $r_1$ and $r_2$ so that the coefficient $c_{1;m}$ first appears in
$T_{p_i}(j_{1;N}^{\ r_i})(z)$ as a factor of $q^d$ for the smallest possible $d$, which leads
to determining $r_i$ from the equation $m-(r_i-1)a_1 \equiv 0 \mod p_i$.  Such a solution
exists provided $p_1$ and $p_2$ are relatively prime to $a_1$.  Without loss of generality, we
may assume that $1 \leq r_1 \leq p_1$.  As in the proof of
Theorem~\ref{genus_zero_integrality}, we impose the further condition that
$(r_1p_1, r_2p_2) = 1$.  For at least one value of $r_2$ in the range
$1 \leq r_2 \leq r_1p_1p_2$, we have that $m-(r_2-1)a_1 \equiv 0 \mod p_2$ and
$r_2p_2 \equiv 1 \mod r_1p_1$, so, in particular, $(r_1p_1,r_2p_2) = 1$.

With the above choices of $r_1$, we determine the first appearance of $c_{1;m}$ in the equation
$$
T_{p_1}(j_{1;N}^{\ r_1})(z) = (j_{1;N}(z))^{r_1p_1} + Q_{1,1}(j_{1;N},j_{2;N}).
$$
On the left-hand-side, $c_{1;m}$ appears as a coefficient of $q^d$ where
$d = (m-(r_1-1)a_1)/p_1$, and, in fact, the coefficient of $q^d$ is equal to an integer plus
$r_1p_1c_{1;m}$.  On the right-hand-side, $c_{1;m}$ first appears as a coefficient of $q^e$ where
$e = m - a_1(r_1p_1-1)$.  We have that $d<e$ when $m>m'$ where
$$
m'=\frac{a_1(r_1p_1^2-r_1-p_1+1)}{p_1-1} = a_1\big(r_1(p_1+1)-1\big) < 2a_1p_1^2.
$$
The resulting expression shows that $c_{1;m}$ can be expressed as a fraction, where the
numerator is a finite sum involving integer multiples of positive powers of $c_{1;k}$ and
$c_{2;k}$ for $k<m$ and denominator equal to $r_1p_1$.  Similarly, by studying the first
appearance of $c_{1;m}$ in the expression
$$
T_{p_2}(j_{1;N}^{\ r_2})(z) = (j_{1;N}(z))^{r_2p_2} + Q_{2,1}(j_{1;N},j_{2;N}),
$$
we obtain an expression showing that $c_{1;m}$ can be expressed as a fraction, where the
numerator is a finite sum involving integer multiples of positive powers of $c_{1;k}$ and
$c_{2;k}$ for $k<m$ and denominator equal to $r_2p_2$.  By induction on $m$, and that
$(r_1p_1,r_2p_2)=1$, we conclude that $c_{1;m}$ is an integer.

Analogously, one studies $c_{2;m}$.  In the four different equations studied, the minimum number
of coefficients needed to be integers in order to initiate the induction is the smallest integer
that is larger than or equal to $a_2(p_2/(p_2-1))(p_1p_2)^2$, which is assumed to hold in the
premise of the theorem.
\end{proof}

\begin{remark}\label{rem:kappa algorithm}\rm
Let us now describe, then implement, an algorithm which will reduce the number of
coefficients which need to be tested for integrality.  For this remark alone, let us
set $a_1 =a_2=1$ if $g_N =0$.

Let $m$ be the lower bound determined in Theorem~\ref{genus_zero_integrality} and
Theorem~\ref{genus_notzero_integrality}.
Let $p_1,\cdots, p_k$ be the set of primes less than $m$ which are
relatively prime to $a_1a_2N$.  For each prime,
let $r_{i,l}$ be the integer in the range $1 \leq r_{i,l} \leq p_i$ satisfying
$m-(r_{i,l}-1)a_l \equiv 0 \mod p_i$, $l\in\{1,2\}$.  Let $S_{m,l}$ denote the set of triples
$$
S_{m,l} = \{(p_i, r_{i,l}, j): \quad j \geq 0 \quad \text{and} \quad
\frac{m-\big((r_{i,l}+jp_i)-1\big)a_l}{p_i} < m-a_l\big((r_{i,l}+jp_i)p_i-1\big)\},
\quad l\in\{1,2\}.
$$
Note that if $r_{i,l}$ is such that $a_l(r_{i,l}(p_i+1)-1)\geq m$, then the set
$S_{m,l}$ does not contain a triple whose prime is $p_i$.

Now let
$$
R_{m,l} = \{(r_{i,l}+jp_i)p_i: \quad (p_i, r_{i,l}, j) \in S_{m,l} \}, \quad l\in\{1,2\}.
$$
Assume that the greatest common divisor of all elements in $R_{m,1}$ is one and that of
$R_{m,2}$ is also one.  Then, by using all the Hecke operators $T_{p_i}$ applied to the
functions $j_{1;N}^{r_{i,1}+jp_i}$ and $j_{2;N}^{r_{i,2}+jp_i}$ and
the arguments from Theorem~\ref{genus_zero_integrality} and
Theorem~\ref{genus_notzero_integrality}, we can determine $c_{1;m}$ and $c_{2;m}$
from lower indexed coefficients and show that $c_{1;m}$ and $c_{2;m}$ are integers
if all lower indexed coefficients are integers.

Although the above observation is too cumbersome to employ theoretically,
it does lead to the following algorithm which can be implemented.
\begin{enumerate}
\item Let $m$ be the lower bound given in Theorem~\ref{genus_zero_integrality} for $g=0$
or Theorem~\ref{genus_notzero_integrality} for $g > 0$.
\item\label{step 2} Construct the sets $S_{m,l}$ and $R_{m,l}$, as described above.
\item If the greatest common divisor of all elements in $R_{m,1}$ is one and that of $R_{m,2}$
is one, too, then replace $m$ by $m-1$ and repeat with Step~\ref{step 2}.
\item If the greatest common divisor of all elements of $R_{m,1}$ or $R_{m,2}$ is greater than
one, let $\kappa = m$.
\end{enumerate}
The outcome of this algorithm yields a reduced number $\kappa$ of coefficients
which need to be tested for integrality in order to conclude that all
coefficients are integers.
\end{remark}
In Tables~\ref{tab:kappa genus zero} and~\ref{tab:kappa genus one} we list the level $N$
and improved bounds on $\kappa$ which were determined by the above algorithm
for all genus zero and genus one moonshine groups.

\begin{table}
\caption{\label{tab:kappa genus zero}Number of expansion coefficients that need to be integer
in order that all coefficients in the $q$-expansion of the Hauptmodul are integer.
Listed is the level $N$ and the value of $\kappa$ according to Remark~\ref{rem:kappa algorithm}
for the genus zero moonshine groups $\Gamma_0(N)^+$.}
$$
\begin{array}{l|cccccccccccccccccccccc}
N      &  1 &  2 &  3 &  5 &  6 &  7 & 10 & 11 & 13 & 14 & 15 & 17 & 19 & 21 & 22 & 23 & 26 & 29 & 30 & 31 & 33 & 34 \\ \hline
\kappa & 19 & 47 & 48 & 19 & 60 & 19 & 75 & 19 & 19 & 47 & 96 & 19 & 19 & 53 & 47 & 19 & 47 & 19 &127 & 19 & 48 & 47 \\
\end{array}
$$

$$
\begin{array}{l|cccccccccccccccccccccc}
N      & 35 & 38 & 39 & 41 & 42 & 46 & 47 & 51 & 55 & 59 & 62 & 66 & 69 & 70 & 71 & 78 & 87 & 94 & 95 &105 &110 &119 \\ \hline
\kappa & 19 & 47 & 48 & 19 &108 & 47 & 19 & 48 & 19 & 19 & 47 & 60 & 48 &181 & 19 & 81 & 48 & 47 & 19 &181 & 89 & 19 \\
\end{array}
$$
\end{table}

\begin{table}
\caption{\label{tab:kappa genus one}Number of expansion coefficients that need to be integer
in order that all coefficients in the $q$-expansions of the field generators are integer.
Listed is the level $N$ and the value of $\kappa$ according to Remark~\ref{rem:kappa algorithm}
for the genus one moonshine groups $\Gamma_0(N)^+$.}
$$
\begin{array}{l|ccccccccccccccccccc}
N      & 37 & 43 & 53 & 57 & 58 & 61 & 65 & 74 & 77 & 79 & 82 & 83 & 86 & 89 & 91 &101 &102 &111 &114 \\ \hline
\kappa &222 &194 &194 &285 &194 &194 &634 &222 &597 &194 &194 &194 &194 &194 &643 &194 &194 &222 &285 \\
\end{array}
$$

$$
\begin{array}{l|ccccccccccccccccccc}
N      &118 &123 &130 &131 &138 &141 &142 &143 &145 &155 &159 &174 &182 &190 &195 &210 &222 &231 &238 \\ \hline
\kappa &194 &194 &634 &194 &194 &194 &194 &271 &382 &382 &194 &194 &643 &382 &634 &599 &222 &597 &599 \\
\end{array}
$$
\end{table}

\section{The Fourier coefficients of the real analytic Eisenstein series}
\label{parabolic Eisenstein}

In this section we derive formulas for coefficients in the Fourier expansion of the real analytic
Eisenstein series $\E_\infty(z,s)$.  First, we derive an intermediate result, expressing the
Fourier coefficients in terms of Ramanujan sums.  Then, we compute coefficients in a closed form.

\subsection{The Fourier expansion of real analytic Eisenstein series on $\Gamma_0(N)^+$}

\begin{lemma}\label{Lemma:Parab Eisen}
The real analytic Eisenstein series $\E_\infty(z,s)$, defined by \eqref{parabol Eis defin} for
$z=x+iy$ and $\Re(s)>1$, admits a Fourier series expansion
\begin{align}\label{Parab.  Four.  exp}
\E_\infty(z,s) = y^s + \varphi_N(s)y^{1-s} + \sum_{m \neq 0} \varphi_N(m,s) W_s(mz),
\end{align}
where $W_s(mz)$ is the Whittaker function given by
$W_s(mz) = 2 \sqrt{\abs m y} K_{s-1/2} (2 \pi \abs m y) e(mx)$
and $K_{s-1/2}$ is the Bessel function.
Furthermore, coefficients of the Fourier expansion (\ref{Parab.  Four.  exp}) are given by
$$
\varphi_N(s) = \frac{s}{s-1}\frac{\xi(2s-1)}{\xi(2s)} \cdot
\prod_{\nu=1}^r \frac{p_\nu^{1-s}+1}{p_\nu^s+1},
$$
where $\xi(s)=\frac12s(s-1)\pi^{-s/2}\Gamma(s/2)\zeta(s)$
is the completed Riemann zeta function and
\begin{align}\label{Phi_N(m,s)}
\varphi_N(m,s) = \pi^s \frac{ m^{s-1}}{\Gamma(s)}
\sum_{v \mid N} v^s N^{-2s} \sum_{n \in \N : (v,n)=1} \frac{a_m((N/v)n)}{n^{2s}},
\end{align}
where we put
$$
a_m(n)=\mu \left( \frac{n}{(\abs m,n)}\right) \frac{\varphi(n)}{\varphi(n/(\abs m,n) ) }.
$$
\end{lemma}
\begin{proof}
The Fourier expansion \eqref{Parab.  Four.  exp} of the real analytic Eisenstein series
$\E_\infty(z,s)$ for $\Re(s)>1$ is a special case of \cite{Iw02}, Theorem 3.4., when the surface
has only one cusp at $i\infty$ with identity scaling matrix.  By $\varphi_N(s)$ we denote the
scattering matrix, evaluated in \cite{JST12} as
$$
\varphi_N(s)=\frac{s}{s-1}\frac{\xi(2s-1)}{\xi(2s)}\cdot D_N(s),
$$
where
$$
D_N(s)=\prod_{\nu=1}^r \frac{p_\nu^{1-s}+1}{p_\nu^s+1}.
$$

If $C_N$ denotes the set of positive left-lower entries of matrices from $\Gamma_0(N)^+$,
then, by \cite{Iw02}, Theorem 3.4.  The Fourier coefficients of \eqref{Parab.  Four.  exp} are
given by
\begin{align}\label{Fourier coeff at m}
\varphi_N(m,s) = \pi^s \frac{\abs m^{s-1}}{\Gamma(s)} \sum_{c\in C_N} \frac{S_N (0,m;c)}{c^{2s}},
\end{align}
where $S_N$ denotes the Kloosterman
sum (see \cite{Iw02}, formula (2.23)) defined, for $c\in C_N$, as
\begin{align}\label{S_N}
S_N(0,m;c) = \sum_{d\mod c:
\big(\begin{smallmatrix} * & * \\ c & d \end{smallmatrix}\big) \in \Gamma_0(N)^+}
e\left(m\frac{d}{c}\right).
\end{align}
In \cite{JST12} we proved that
$$
C_N=\{(N/\sqrt{v})n: \quad v\mid N, n\in\N\}.
$$
For a fixed $c=(N/\sqrt{v})n \in C_N$, with $v\mid N$ and $n\in\N$ arbitrary,
we can take $e=v$ in the definition of $\Gamma_0(N)^+$ to deduce that matrices
from $\Gamma_0(N)^+$ with left lower entry $c$ are given by
$$
\begin{pmatrix} \sqrt{v}a & b/\sqrt{v} \\ \frac{N}{v}\sqrt{v}n & \sqrt{v}d \end{pmatrix}
$$
for some integers $a$, $b$ and $d$ such that $vad-(N/v)bn=1$.
Since $N$ is square-free, this equation has a solution if and only if
$(v,n)=1$ and $(d,(N/v)n)=1$.

Therefore, for $m\in\Z\setminus\{0\}$, and fixed
$c= (N/\sqrt{v})n \in C_N$ equation \eqref{S_N} becomes
$$
S_N(0, m; (N/\sqrt{v})n)
= \sum_{\substack{1\leq d<(N/v)n \\ (d, (N/v)n)=1}} e\left(\frac{dm}{(N/v)n}\right).
$$
This is a Ramanujan sum, which can be evaluated using
formula (3.3) from \cite{IK04}, page 44 to get
$$
S_N(0, m; (N/\sqrt{v})n) = \mu \left( \frac{(N/v) n}{(\abs m, (N/v)n)} \right)
\frac{\varphi((N/v) n)}{\varphi((N/v)n/(\abs m, (N/v)n) ) }.
$$
For $n\in\N$ element $c= (N/\sqrt{v})n$ belongs to $C_N$ if and only if $v\mid N$
and $(n,v)=1$, hence equation \eqref{Fourier coeff at m} becomes
$$
\varphi_N(m,s) = \pi^s \frac{\abs m^{s-1}}{\Gamma(s)}\sum_{v\mid N} \sum_{(n,v)=1}
\frac{1}{((N/\sqrt{v})n)^{2s}} \cdot \mu \left( \frac{(N/v) n}{(\abs m, (N/v)n)} \right)
\frac{\varphi((N/v) n)}{\varphi((N/v)n/(\abs m, (N/v)n) ) },
$$
which proves \eqref{Phi_N(m,s)}.
\end{proof}

\subsection{Computation of Fourier coefficients}

In this subsection we will compute coefficients \eqref{Phi_N(m,s)} in closed form.  We will
assume that $m>0$ and incorporate negative terms via the identity $\varphi_N(-m,s)=\varphi_N(m,s)$.

Let $p$ and $q$ to denote prime numbers, and let $\alpha_p \in \Z_{\geq0}$  denote the
highest power of $p$ that divides $m$, i.e.\ the number $\alpha_p$ is such that
$p^{\alpha_p} \mid m$ and $p^{\alpha_p+1} \nmid m$.
For any prime $p$, set
$$
F_p(m,s)=\left( 1-\frac1{p^s}\right)\left( 1+\frac1{p^{s-1}}+\ldots
+\frac1{p^{\alpha_p(s-1)}}\right)
$$
and
$$
D_N(m,s)=\sum_{v \mid N}v^{-s} \cdot \prod_{p \mid v} F_p(m,2s)^{-1}
\cdot \prod_{p \mid (N/v)} (1-F_p(m,2s)^{-1}).
$$

\begin{theorem}\label{Th Fourier coeff}
Assume $N$ is square-free and let $\varphi_N(m,s)$ denote the $m$-th coefficient in the
Fourier series expansion \eqref{Parab.  Four.  exp} of the Eisenstein series.
Then the coefficients $\varphi_N(m,s)$ can be meromorphically continued
from the half plane $\Re(s)>1$ to the whole complex plane by the equation
$$
\varphi_N(m,s)= \pi^s\frac{m^{s-1}}{\Gamma(s)}
\cdot \frac{\sigma_{1-2s}(m)}{\zeta(2s)}\cdot D_N(m,s).
$$
\end{theorem}
\begin{proof}
We employ the equation \eqref{Phi_N(m,s)} and observe that coefficients $a_m(n)$ are
multiplicative.  Therefore, it is natural to investigate the associated $L$-series
$$
L_m(s)= \sum_{n=1}^\infty \frac{a_m(n)}{n^s}
$$
defined for $\Re(s)>1$.  The series $L_m(s)$ is defined in \cite{Ko09}, Lemma 2.2 for $\Re(s)>1$,
where its analytic continuation to the whole complex plane is given by the formula
$L_m(s) = \sigma_{1-s}(m)/\zeta(s)$.

The multiplicativity of coefficients $a_m(n)$ implies that for $\Re(s)>1$ one has the
Euler product
$$
L_m(s) = \prod_p E_p(m,s)
= \prod_p \left( 1+ \frac{a_m(p)}{p^s}+ \frac{a_m(p^2)}{p^{2s}}+ \ldots\right),
$$
where $E_p(m,s)$ stands for the $p$-th Euler factor, i.e.
\begin{align}\label{F_p formula}
E_p(m,s)= \sum_{\nu=0}^\infty \frac{a_m(p^\nu)}{p^{\nu s}}
\quad \text{for $\Re(s) >1$.}
\end{align}
By the computations presented in \cite{Ko09}, page 1137, one has
$$
E_p(m,s)=\left( 1-\frac1{p^s}\right)\left( 1+\frac1{p^{s-1}}+ \ldots
+\frac1{p^{\alpha_p (s-1)}}\right),
$$
hence, $E_p(m,s)=F_p(m,s)$.  Therefore,
\begin{align}\label{coprime basic sum}
\sum_{n \in \N : (N,n)=1} \frac{a_m(n)}{n^s}
= \prod_{p \nmid N} E_p(m,s)
= \frac{\sigma_{1-s}(m)}{\zeta(s)} \cdot \prod_{p \mid N} F_p(m,s)^{-1}.
\end{align}

Let $v \mid N$ be fixed.  In order to compute the sum
$$
\sum_{n \in \N : (v,n)=1} \frac{a_m((N/v)n)}{n^{2s}}
$$
using
equation \eqref{coprime basic sum}, we use the fact that $N$ is square-free,
so $(v, N/v)=1$, hence every integer $n$ coprime to $v$ can be represented as
\begin{align}\label{rep n}
n= \prod_{q \mid (N/v)} q^{n_q} \cdot k, \text{ where } n_q \geq0, \ (k,v)=1
\text{ and } (k,q)=1 \text{ for all } q \mid (N/v).
\end{align}
Since $(k,v)=1$ and $(k,q)=1$ for every $q \mid (N/v)$, we deduce that
actually $(k,N)=1$.
Therefore
\begin{align}\label{rep set}
\{n\in\N : (v,n)=1\} = \bigcup_{q \mid (N/v)} \bigcup_{n_q=0}^\infty
\left( \left\{ \prod_{p \mid (N/v)} p^{n_p} \cdot k : \quad (k,N)=1 \right\}
\right).
\end{align}
Let $q_1,\ldots,q_l$ denote the set of all prime divisors of $N/v$.
The multiplicativity of coefficients $a_m$ implies that for $n$ given by
\eqref{rep n} one has
$$
a_m\left( \frac{N}{v}n \right)= \prod_{\nu=1}^l a_m( q_\nu^{n_{q_\nu}+1}) a_m (k);
$$
hence, by \eqref{rep set} one gets that
$$
\sum_{n \in \N : (v,n)=1} \frac{a_m((N/v)n)}{n^{2s}}
= \sum_{n_{q_1}=0}^\infty \cdots \sum_{n_{q_l}=0}^\infty
\prod_{\nu=1}^l \frac{a_m( q_\nu^{n_{q_\nu}+1})}{(q_\nu^{n_{q_\nu}})^{2s}}
\sum_{(k,N)=1}\frac{a_m(k)}{k^{2s}}.
$$
The identity \eqref{coprime basic sum} is now applied to the inner sum in the above
equation, from which we deduce that
$$
\sum_{n \in \N : (v,n)=1} \frac{a_m((N/v)n)}{n^{2s}}
= \frac{\sigma_{1-2s}(m)}{\zeta(2s)}
\cdot \prod_{p \mid N} F_p(m,2s)^{-1}
\cdot \sum_{n_{q_1}=0}^\infty \frac{a_m( q_1^{n_{q_1}+1})}{(q_1^{n_{q_1}})^{2s}}
\cdots \sum_{n_{q_l}=0}^\infty \frac{a_m( q_l^{n_{q_l}+1})}{(q_l^{n_{q_l}})^{2s}}
$$
for $\Re(s)>1$.

For all $\nu=1,\ldots,l$ one then has
$$
\sum_{n_{q_\nu}=0}^\infty \frac{a_m( q_\nu^{n_{q_\nu}+1})}{(q_\nu^{n_{q_\nu}})^{2s}}
= q_\nu^{2s} \cdot \sum_{n_{q_\nu}=0}^\infty
\frac{a_m( q_\nu^{n_{q_\nu}+1})}{(q_\nu^{n_{q_\nu}+1})^{2s}}= q_\nu^{2s}
\cdot (F_{q_\nu}(m,2s)-1),
$$
by \eqref{F_p formula}.
Therefore,
\begin{multline}\label{phi final sum}
\sum_{n \in \N : (v,n)=1} \frac{a_m((N/v)n)}{n^{2s}}
= \frac{\sigma_{1-2s}(m)}{\zeta(2s)} \cdot \prod_{p \mid N} F_p(m,2s)^{-1}
\cdot \left(\frac{N}{v} \right)^{2s} \prod_{q \mid (N/v)} (F_q(m,2s)-1) \\
= N^{2s}v^{-2s}\frac{\sigma_{1-2s}(m)}{\zeta(2s)}
\cdot \prod_{p \mid v} F_p(m,2s)^{-1}
\cdot \prod_{q \mid (N/v)} (1-F_q(m,2s)^{-1}).
\end{multline}

Substituting \eqref{phi final sum} into \eqref{Phi_N(m,s)} completes the proof.
\end{proof}

\begin{proposition}
For a square-free integer $N=\prod_{\nu=1}^r p_\nu$,\ \ $D_N(m,s)$ is
multiplicative in $N$,
$$
D_N(m,s)=\prod_{\nu=1}^r \left( 1- \left(1-\frac1{p_\nu^s} \right)
F_{p_\nu}(m,2s)^{-1}\right)=\prod_{\nu=1}^r D_{p_\nu}(m,s).
$$
Furthermore, for a prime $p$ and $m\in\Z_{>0}$
$$
D_p(m,s)= 1-\frac{p^s}{p^s+1} \left( \sigma_{1-2s}(p^{\alpha_p})\right)^{-1},
$$
where $\alpha_p\in\Z_{\geq0}$ is the highest power of $p$ dividing $m$.
\end{proposition}
\begin{proof}
We apply induction on number $r$ of prime factors of $N$.
For $r=1$ the statement is immediate.

Assume that the statement is true for all numbers $N$ with $r$ distinct prime
factors and let $N=\prod_{\nu=1}^{r+1}p_\nu$.  Since
$\{v: v\mid N\}=\{v: v\mid \prod_{\nu=1}^r p_\nu\}
\cup \{v\cdot p_{r+1}: v \mid \prod_{\nu=1}^r p_\nu\}$,
the statement is easily deduced from the definition of $D_N(m,s)$ and the inductive assumption.

The second statement follows trivially from the properties of the function $\sigma_s(m)$
and the definition of the Euler factor $F_p$.
\end{proof}

\begin{corollary}
For any square-free $N$, the groups $\Gamma_0(N)^+$ have no residual eigenvalues
$\lambda\in[0,1/4)$ besides the obvious one $\lambda_0=0$.
\end{corollary}
\begin{proof}
From Lemma~\ref{Lemma:Parab Eisen} with $m=0$, we have a formula for the constant
term $\varphi_N(s)$ in the Fourier expansion of the real analytic Eisenstein series, namely
$$
\varphi_N(s)=\frac{s}{s-1}\frac{\xi(2s-1)}{\xi(2s)}\cdot D_N(s);
$$
see also Lemma 5 of \cite{JST12}.  It is elementary to see that $\varphi_N(s)$
has no poles in the interval $(1/2,1)$.  From the spectral theory, one has that the
residual eigenvalues $\lambda$ of the hyperbolic Laplacian such that $\lambda \in (0,1/4)$
correspond to poles $s$ of the real analytic Eisenstein series with $s \in (1/2,1)$ with
the relation $s(1-s) = \lambda$.  Since poles of real analytic Eisenstein series $\E_\infty(z,s)$
are exactly the poles of $\varphi_N(s)$, the statement is proved.
\end{proof}

\section{Kronecker's limit formula for $\Gamma_0(N)^+$}\label{Kronecker limit}

In this section we derive Kronecker's limit formula for the Laurent series expansion of the
real analytic Eisenstein series.  First, we will prove the following, intermediate result.

\begin{proposition}\label{pre-Kroneck limit}
For a square-free integer $N=\prod_{\nu=1}^r p_\nu$, we have the Laurent expansion
\begin{align}\label{pre_kronecker limit f-la}
\E_\infty(z,s)
=\frac{C_{-1,N}}{(s-1)} + C_{0,N} + C_{1,N} \log y + y
+ \frac{6}{\pi} \sum_{m=1}^\infty
\sigma_{-1}(m) D_N(m,1) ( e(mz) + e(-m\overline{z}) ) + O(s-1),
\end{align}
as $s \to 1$, where
$$
C_{-1,N} = -C_{1,N} = \frac{3\cdot2^r}{\pi\sigma(N)}
$$
and
$$
C_{0,N}= \frac{3\cdot2^r}{\pi\sigma(N)}
\left( -\frac12\sum_{\nu=1}^r \frac{1+3p_\nu}{1+p_\nu} \log p_\nu
+ 2-24\zeta'(-1)-2\log4\pi \right).
$$
\end{proposition}
\begin{proof}
We start with the formula \eqref{Parab.  Four.  exp} for the
Fourier expansion of real analytic Eisenstein series at the cusp $i\infty$.

Let $\varphi$ denote the scattering matrix for $\PSL(2,\Z)$.
Then, one has $\varphi_N(s) = \varphi(s) D_N(s)$.
We will use this fact in order to deduce the first two terms in the Laurent
series expression of $\E_\infty(z,s)$ around $s=1$.

By the classical Kronecker limit formula for $\PSL(2,\Z)$ the Laurent
series expansion of $\varphi(s) y^{1-s}$ at $s=1$ is given by
(see e.g.\ \cite{Ku2001}, page 228)
$$
\frac\pi3 \varphi(s) y^{1-s} = \frac1{s-1} + A - \log y + O(s-1),
$$
where $A=2-24\zeta'(-1)-2\log4\pi$.

The function $D_N(s)$ is holomorphic in any neighborhood of $s=1$, hence
$D_N(s)= D_N(1) + D_N'(1)(s-1) + O(s-1)^2$, as $s \to 1$.
Therefore,
$$
\frac\pi3 \varphi_N(s) y^{1-s}
= \frac{D_N(1)}{s-1} + D_N'(1)+(A-\log y)D_N(1) + O(s-1),
$$
as $s\to1$.

In order to get the residue and the constant term steaming from $\varphi_N$
it is left to compute values of $D_N$ and $D_N'$ at $s=1$,
$$
D_N(1)= 2^r \prod_{\nu=1}^r (1+p_\nu)^{-1}=\frac{2^r}{\sigma(N)}
\ \text{ and }\ D_N'(1)
= -\frac{2^{r-1}}{\sigma(N)} \sum_{\nu=1}^r \frac{1+3p_\nu}{1+p_\nu} \log p_\nu.
$$
This, together with \eqref{Parab.  Four.  exp} and Theorem
\ref{Th Fourier coeff} completes the proof.
\end{proof}

Since the series $\E_\infty(z,s)$ always has a pole at $s=1$ with residue $1/\vol(X_N)$, from
the above proposition, we easily deduce a simple formula for the volume of the surface $X_N$
in terms of the level of the group, namely
$$
\vol(X_N)= \frac{\pi \sigma(N)}{3 \cdot 2^r}.
$$

\begin{theorem}\label{thm: Kronecker limit}
For a square-free integer $N=\prod_{\nu=1}^r p_\nu$, we have the asymptotic expansion
$$
\E_\infty(z,s)
=\frac{C_{-1,N}}{(s-1)} + C_{0,N}
-\frac1{\vol X_N} \log\left(\sqrt[2^r]{\prod_{v \mid N}
\abs{ \eta(vz)}^4} \cdot \Im(z) \right)+ O(s-1),
$$
as $s \to 1$, where $C_{-1,N}$and $C_{0,N}$ are defined in Proposition
\ref{pre-Kroneck limit} and $\eta$ is the Dedekind eta function defined
for $z \in \h$ by
$$
\eta(z)= e(z/24) \prod_{n=1}^\infty (1-e(nz))
\quad \text{where} \quad
e(z) = e^{2\pi i z}.
$$
\end{theorem}
\begin{proof}
From the definition of Euler factors $F_p(m,s)$, using the fact that
$\sigma_{-1}(m)=\frac{\sigma(m)}m$, we get that
$$
F_p(m,2)=
\left( 1-\frac1{p^2}\right)\frac{\sigma(p^{\alpha_p})}{p^{\alpha_p}}.
$$
Therefore, a simple computation implies that
$$
D_{p_\nu}(m,1)=\frac1{p_\nu+1}\left(1+\frac{\sigma(p_\nu^{\alpha_{p_\nu}})-1}{
\sigma(p_\nu^{\alpha_{p_\nu}})} \right),
$$
hence
$$
D_N(m,1)= \prod_{\nu=1}^rD_{p_\nu}(m,1)
= \frac1{\sigma(N)}\prod_{\nu=1}^r \left(1+\frac{\sigma(p_\nu^{\alpha_{p_\nu}})-1}{
\sigma(p_\nu^{\alpha_{p_\nu}})} \right)= \frac1{\sigma(N)} \sum_{k=0}^r
\sum_{1\leq i_1<\ldots<i_k\leq r} P(i_1,\ldots,i_k),
$$
where we put $P(i_1,\ldots,i_k)=1$ for $k=0$, and for $k=1,\ldots,r$
$$
P(i_1,\ldots,i_k)=\left(\frac{\sigma(p_{i_1}^{\alpha_{p_{i_1}}})-1}{
\sigma(p_{i_1}^{\alpha_{p_{i_1}}})} \right) \cdots \left(\frac{
\sigma(p_{i_k}^{\alpha_{p_{i_k}}})-1}{\sigma(p_{i_k}^{\alpha_{p_{i_k}}})} \right).
$$
Every $m\geq 1$ can be written as
$m=p_1^{\alpha_1} \cdots p_r^{\alpha_r} \cdot l$, where $\alpha_i \geq 0$
and $(l,N)=1$.
Therefore, we may write the sum over $m$ in \eqref{pre_kronecker limit f-la} as
\begin{multline}\label{complete sum}
\frac1{\sigma(N)} \sum_{k=0}^r \sum_{1\leq i_1<\ldots<i_k\leq r}
\Big( \sum_{(l,N)=1} \sum_{\alpha_1=0}^\infty \cdots \sum_{\alpha_r=0}^\infty
\frac{\sigma(p_1^{\alpha_1})\cdots \sigma(p_r^{\alpha_r})
\cdot \sigma(l)}{ p_1^{\alpha_1} \cdots p_r^{\alpha_r} \cdot l}
P(i_1,\ldots,i_k) \cdot \\ \cdot (e(p_1^{\alpha_1} \cdots p_r^{\alpha_r}
\cdot l z) + e(-p_1^{\alpha_1} \cdots p_r^{\alpha_r} \cdot l \overline{z}))
\Big).
\end{multline}
When $k=0$ the above sum obviously reduces to
\begin{align}\label{eta sum}
\sum_{m=1}^\infty \frac{\sigma(m)}m ( e(mz) + e(-m\overline{z}) )
= - \log\abs{\eta(z) }^2 - \frac{\pi}{6} y,
\end{align}
as deduced from the proof of the classical $\PSL(2,\Z)$ Kronecker limit
formula.

Now, we will take $k\in\{1,\ldots,r\}$ and compute one term in the
sum \eqref{complete sum}.
Without loss of generality, in order to ease the notation, we may assume
that $i_1=1,\ldots, i_k=k$ and compute the term with $P(1,\ldots,k)$.
Later, we will take the sum over all indices $1\leq i_1 < \ldots <i_k\leq r$.
First, we observe that $P(1,\ldots,k)=0$ if $\alpha_i=0$, for any
$i\in{1,\ldots,k}$ and
$$
P(1,\ldots,k)=\frac{p_1\sigma(p_1^{\alpha_1 -1})
\cdots p_k\sigma(p_k^{\alpha_k -1}) }{\sigma(p_1^{\alpha_1})
\cdots \sigma(p_k^{\alpha_k})},
$$
if all $\alpha_i \geq 1$, for $i=1,\ldots,k$.
Furthermore, for $\alpha_i \geq 1$, $i=1,\ldots,k$ one has
$$
\frac{\sigma(p_1^{\alpha_1})\cdots \sigma(p_r^{\alpha_r})
\cdot \sigma(l)}{ p_1^{\alpha_1} \cdots p_r^{\alpha_r} \cdot l} P(1,\ldots,k)
=\frac{\sigma(p_1^{\alpha_1 -1})\cdots \sigma(p_k^{\alpha_k -1})
\cdot \sigma(p_{k+1}^{\alpha_{k+1}}) \cdots \sigma(p_r^{\alpha_r})
\cdot \sigma(l)}{ p_1^{\alpha_1 -1} \cdots p_k^{\alpha_k -1}
\cdot p_{k+1}^{\alpha_{k+1}} \cdots p_r^{\alpha_r} \cdot l}
$$
Therefore,
\begin{multline*}
\sum_{(l,N)=1} \sum_{\alpha_1=0}^\infty \cdots \sum_{\alpha_r=0}^\infty
\frac{\sigma(p_1^{\alpha_1})\cdots \sigma(p_r^{\alpha_r})
\cdot \sigma(l)}{ p_1^{\alpha_1} \cdots p_r^{\alpha_r} \cdot l}
P(1,\ldots,k) \cdot (e(p_1^{\alpha_1} \cdots p_r^{\alpha_r} \cdot l z)
+ e(-p_1^{\alpha_1} \cdots p_r^{\alpha_r} \cdot l \overline{z}))\\
=\sum_{(l,N)=1} \sum_{\alpha_1 =1}^\infty \cdots \sum_{\alpha_k =1}^\infty
\sum_{\alpha_{k+1}=0}^\infty \cdots \sum_{\alpha_r=0}^\infty
\frac{\sigma(p_1^{\alpha_1 -1})\cdots \sigma(p_k^{\alpha_k -1})
\cdot \sigma(p_{k+1}^{\alpha_{k+1}}) \cdots \sigma(p_r^{\alpha_r})
\cdot \sigma(l)}{ p_1^{\alpha_1 -1} \cdots p_k^{\alpha_k -1}
\cdot p_{k+1}^{\alpha_{k+1}} \cdots p_r^{\alpha_r} \cdot l} \cdot
\\ \cdot\Big( e(p_1 \cdots p_k z)^{p_1^{\alpha_1-1} \cdots p_k^{\alpha_k-1}
\cdot p_{k+1}^{\alpha_{k+1}} \cdots p_r^{\alpha_r} \cdot l}
+ e(-p_1 \cdots p_k \overline{z})^{p_1^{\alpha_1-1} \cdots p_k^{\alpha_k-1}
\cdot p_{k+1}^{\alpha_{k+1}} \cdots p_r^{\alpha_r} \cdot l} \Big).
\end{multline*}
Since every integer $m\geq 1$ can be represented as
$m=p_1^{\alpha_1-1} \cdots p_k^{\alpha_k-1}
\cdot p_{k+1}^{\alpha_{k+1}} \cdots p_r^{\alpha_r} \cdot l$, where
$\alpha_1,\ldots,\alpha_k \geq 1$; $\alpha_{k+1}, \ldots, \alpha_r \geq 0$
and $(l,N)=1$ we immediately deduce that
\begin{multline*}
\sum_{(l,N)=1} \sum_{\alpha_1=0}^\infty \cdots \sum_{\alpha_r=0}^\infty
\frac{\sigma(p_1^{\alpha_1})\cdots \sigma(p_r^{\alpha_r})
\cdot \sigma(l)}{ p_1^{\alpha_1} \cdots p_r^{\alpha_r} \cdot l}
P(1,\ldots,k) \cdot (e(p_1^{\alpha_1} \cdots p_r^{\alpha_r} \cdot l z)
+ e(-p_1^{\alpha_1} \cdots p_r^{\alpha_r} \cdot l \overline{z})) \\
=\sum_{m=1}^\infty \frac{\sigma(m)}m (e(p_1 \cdots p_k z)^m
+ e(-p_1 \cdots p_k \overline{z})^m)
= - \log\abs{\eta(p_1\cdots p_k z) }^2 -\frac{\pi}{6} p_1\cdots p_k y,
\end{multline*}
by \eqref{eta sum}.

Let us now sum over all $k=0,\ldots,r$ and the sum over all indices
$1\leq i_1<\ldots<i_k\leq r $, which is equivalent to taking the sum over all
$v \mid N$.  The sum \eqref{complete sum} then becomes
$$
\frac1{\sigma(N)} \sum_{v \mid N} \left( -\log\abs{\eta(vz) }^2
-\frac{\pi}{6} v y \right)
= -\frac1{\sigma(N)} \log\left(\prod_{v \mid N}
\abs{ \eta(vz)}^2 \right) - \frac{\pi}{6}y,
$$
since $\sum_{v \mid N} v =\sigma(N)$.
Therefore, we get that
\begin{multline*}
C_{1,N} \log y + y + \frac{6}{\pi} \sum_{m=1}^\infty
\sigma_{-1}(m) D_N(m,1) ( e(mz) + e(-m\overline{z}) )
= -\frac{6}{\pi \sigma(N)} \log\left(\prod_{v \mid N}
\abs{ \eta(vz)}^2 \right) - \frac1{\vol X_N}\log(y) \\
= \frac{-1}{\vol X_N} \log\left(\sqrt[2^r]{\prod_{v \mid N}
\abs{ \eta(vz)}^4} \cdot \Im(z) \right).
\end{multline*}
The proof is complete.
\end{proof}

\begin{remark}
Since the number of divisors of $N=p_1\cdots p_r$ is $2^r$, the expression
$$
\sqrt[2^r]{\prod_{v \mid N} \abs{ \eta(vz)}^4}
$$
is a geometric mean.
\end{remark}

When $N=1$, Theorem~\ref{thm: Kronecker limit} amounts to the classical Kronecker limit formula.

From Theorem~\ref{thm: Kronecker limit} and functional equation
$\E_\infty(z,s)
=\varphi_N(s)\E_\infty(z,1-s)$
for real analytic Eisenstein series we can reformulate the Kronecker limit formula
as asserting an expansion for $\E_\infty(z,s)$ as $s\to 0$.

\begin{proposition}
For $z\in\h$, the real analytic Eisenstein series
$\E_\infty(z,s)$ admits a Taylor series
expansion of the form
$$
\E_\infty(z,s)
= 1+ (\log\left(\sqrt[2^r]{\prod_{v \mid N} \abs{ \eta(vz)}^4}
\cdot \Im(z) \right))\cdot s + O(s^2), \text{ as } s\to0.
$$
\end{proposition}

An immediate consequence of Theorem~\ref {thm: Kronecker limit} and the above
proposition is the fact that the function
\begin{align}\label{group inv}
\sqrt[2^r]{\prod_{v \mid N} \abs{ \eta(vz)}^4} \cdot \Im(z)
\end{align}
is invariant under the action of the group $\Gamma_0(N)^+$.
Using the fact that the eta function is non-vanishing on $\h$, we may
deduce the stronger statement.

\begin{proposition}\label{character prop}
There exists a character
$\eps_N(\gamma)$ on $\Gamma_0(N)^+$ such that
\begin{align}\label{product of etas}
\prod_{v \mid N} \eta(v \gamma z)
=\eps_N(\gamma)(cz+d)^{2^{r-1}}\prod_{v \mid N} \eta(v z),
\end{align}
for all $\gamma \in \Gamma_0(N)^+$.
\end{proposition}
\begin{proof}
Let
$$
\gamma= \begin{pmatrix} * & * \\ c & d \end{pmatrix} \in \Gamma_0(N)^+.
$$
Since the expression \eqref{group inv} is a group invariant we have
$$
\sqrt[2^r]{\prod_{v \mid N} \abs{ \eta(v \gamma z)}^4}
\cdot \frac{\Im(z)}{(cz+d)^2} = \sqrt[2^r]{\prod_{v \mid N}
\abs{ \eta(vz)}^4} \cdot \Im(z).
$$
Dividing by $\Im (z)\neq 0$ and raising the expression to the $2^r/4$ power, we get that
$$
\abs{\prod_{v \mid N} \eta(v \gamma z)} =\abs{cz+d}^{2^{r-1}} \abs{\prod_{v \mid N} \eta(v z)},
$$
hence
$$
\prod_{v \mid N} \eta(v \gamma z)=f(\gamma,z)(cz+d)^{2^{r-1}}\prod_{v \mid N} \eta(v z),
$$
for some function $f$ of absolute value $1$.

Since the function $\eta$ is a holomorphic function, non-vanishing on $\h$, for a fixed
$\gamma \in \Gamma_0(N)^+$ the function $f(\gamma, z)$ is a holomorphic function in
$z$ on $\h$ of absolute value $1$ on $\h$, hence it is constant, as a function of $z$.
Therefore, $f(z,\gamma)=\eps_N(\gamma)$, for all $z\in \h$ where
$\abs{\eps_N(\gamma)}=1$.

It is left to prove that $\eps_N(\gamma_1 \gamma_2)
=\eps_N(\gamma_1) \eps_N(\gamma_2),$ for all
$\gamma_1, \gamma_2 \in \Gamma_0(N)^+$.
This is an immediate consequence of formula \eqref{product of etas}.
\end{proof}

The character $\eps_N(\gamma)$ is the analogue of the classical Dedekind sum.

In the case when the genus of the surface $X_N$ is equal to zero, the group
$\Gamma_0(N)^+$ is a free product of finite, cyclic groups, hence the
character $\eps_N(\gamma)$ is finite in the sense that there exists an
integer $\ell_N$ such that $\eps_N(\gamma)^{\ell_N} =1$ for all
$\gamma \in \Gamma_0(N)^+$.

Irrespective of the genus of the surface $X_N$, we will prove that the
character $\eps_N(\gamma)$ is a certain 24th root of unity.

\begin{theorem}\label{Delta_N}
Let $N=p_1 \cdots p_r$.  Let us define the constant $\ell_N$ by
$$
\ell_N = 2^{1-r}\lcm\Big(4,\ 2^{r-1}\frac{24}{(24,\sigma(N))}\Big)
$$
Then, the function
$$
\Delta_N(z):=\left( \prod_{v \mid N} \eta(v z) \right)^{\ell_N}
$$
is a weight $k_N=2^{r-1} \ell_N$ holomorphic form on $\Gamma_0(N)^+$ vanishing only at the cusp.
\end{theorem}
\begin{proof}
The vanishing of the function $\Delta_N(z)$ at the cusp $i\infty$ only
is an immediate consequence of properties of the Dedekind eta function.
Therefore, it is left to prove that $\Delta_N(z)$ is a weight $k_N=2^{r-1} \ell_N$ holomorphic
form on $\Gamma_0(N)^+$.

We begin with the decomposition
$\Gamma_0(N)^+ =\Gamma_0(N) \bigcup\limits_{v \mid N, v>1}\Gamma_0(N,v) $, where
$$
\Gamma_0(N,v)=
\left\{ \begin{pmatrix} a\sqrt{v} & b/\sqrt{v} \\ Nc/\sqrt{v} & d\sqrt{v} \end{pmatrix}
\in \SL(2,\R): \quad a,b,c,d \in \Z \right\},
$$
see \cite{Macl80}, page 147, with a slightly different notation; elements $b/\sqrt{v}$ and
$Nc/\sqrt{v}$ are interchanged.  For each $v \mid N$ and $v>1$, there exists an
order two element of $\Gamma_0(N)^+$, say, $\tau_v$, called Atkin-Lehner involution,
(see \cite{AtLeh70}, Lemma 8) such that
$\Gamma_0(N,v) = \Gamma_0(N) \tau_v$.  Therefore, for an arbitrary $\gamma \in \Gamma_0(N)^+$
either $\gamma \in \Gamma_0(N)$ or there exists $v\mid N$, $v>1$ and an element
$\gamma_v \in \Gamma_0(N)$ such that $\gamma=\gamma_v \tau_v$.

Let $\eps_N$ denote the character defined in Proposition~\ref{character prop}.  For any
$\gamma = \big(\begin{smallmatrix} * & * \\ c & d \end{smallmatrix}\big) \in \Gamma_0(N)^+$
and $z\in\h$, by Proposition~\ref{character prop} we have
$$
\Delta_N(\gamma z) = \eps_N(\gamma) ^{\ell_N} (cz+d)^{2^{r-1}\ell_N} \Delta_N(z).
$$
Therefore, in order to prove the theorem we need to prove that $\eps_N(\gamma) ^{\ell_N} =1$
for all $\gamma \in \Gamma_0(N)^+$.
Since $\eps_N$ is multiplicative, it is sufficient to prove that $\eps_N(\tau_v) ^{\ell_N} =1$,
for all $v\mid N$, $v>1$ and $\eps_N(\gamma) ^{\ell_N} =1$ for all $\gamma\in\Gamma_0(N)$.

Let $\tau = \big(\begin{smallmatrix} a_{\tau} & b _{\tau} \\ c_{\tau} & d_{\tau}
\end{smallmatrix}\big) \in \Gamma_0(N)^+ \setminus\{\Id\}$
be an arbitrary element of order two.  Then there exists $z_{\tau}\in\h$ such that
$\tau z_{\tau} =z_{\tau}$ and $\tau^2 =\Id$, where $\Id$ stands for the identity.
Obviously, $\Delta_N(\tau z_{\tau}) = \Delta_N(z_{\tau})$.  From
Proposition~\ref{character prop} we have that
$$
\Delta_N(\tau z_{\tau})=\eps_N(\tau)^{\ell_N} (c_{\tau} z_{\tau} +d_{\tau})^{2^{r-1}\ell_N}
\Delta_N(z_{\tau}).
$$
Since $\Delta_N(z_{\tau}) \neq 0$, we immediately deduce that
$$
\eps_N(\tau)^{\ell_N} (c_{\tau} z_{\tau} +d_{\tau})^{k_N} =1.
$$
The condition $\tau^2=\Id$, $\tau \neq \Id$ implies that $a_{\tau} + d_{\tau} =0$.  From
$a_{\tau}d_{\tau}-b_{\tau}c_{\tau}=1$ and $\tau z_{\tau} =z_{\tau}$ it is easily deduced that
$(c_{\tau} z_{\tau} +d_{\tau}) =i$, hence $\eps_N(\tau)^{\ell_N} (i)^{k_N} =1$.  By the
definition of $\ell_N$, the number $k_N=2^{r-1} \ell_N$ is always divisible by $4$, hence
$(i)^{k_N} =1$.  Therefore, $\eps_N(\tau)^{\ell_N} =1$.

In order to complete the proof of the theorem, it is left to prove that
$\eps_N(\gamma)^{\ell_N} =1$, for an arbitrary
$\gamma = \big(\begin{smallmatrix} * & * \\ c & d \end{smallmatrix}\big) \in \Gamma_0(N)$.
For this, we apply results of \cite{Raji06}, chapter 2.2.3, pages 21--23 with
$f_1 = \Delta_N$; $g=2^{r-1}$; $\delta_l =v$, $r_{\delta_l} = \ell_N$, for all $\delta_l =v$
and $k=k_N$.  By the definition of $\ell_N$, we see that $\ell_N \sigma(N) \equiv 0 \mod 24$,
therefore, conditions (2.11) and (2.12) of \cite{Raji06} are fulfilled, so then
$$
\eps_N(\gamma)^{\ell_N}= \chi(d)
= \left( \frac{(-1)^{k_N}}{d}\right) \prod_{v \mid N} \left(\frac{v}{d} \right)^{\ell_N},
$$
where $\left( \frac{v}{d}\right)$ denotes the Jacobi symbol.
The multiplicativity of the Jacobi symbol implies that
$$
\prod_{v \mid N} \left(\frac{v}{d} \right)= \prod_{\nu=1}^r \left(\frac{p_\nu}{d} \right)^{2^{r-1}},
$$
hence
$$
\eps_N(\gamma)^{\ell_N}= \chi(d)
= \left( \left( \frac{(-1)}{d}\right) \prod_{\nu=1}^r \left(\frac{p_\nu}{d} \right)\right)^{k_N}=1,
$$
since $k_N$ is even.  The proof is complete.
\end{proof}

In Tables~\ref{tab:0} and~\ref{tab:1} we list the values of degree $\ell_N$ and weight
$k_N$ for all genus zero and genus one moonshine groups.

\begin{table}
\caption{\label{tab:0}The degree $\ell_N$ and the weight $k_N$ of the modular form $\Delta_N$ on
$\Gamma_0(N)^+$ for all moonshine groups of genus zero.
Listed are the level $N$, and the values of degree $\ell_N$ and weight $k_N$.}
$$
\begin{array}{l|cccccccccccccccccccccc}
N      &  1 &  2 &  3 &  5 &  6 &  7 & 10 & 11 & 13 & 14 & 15 & 17 & 19 & 21 & 22 & 23 & 26 & 29 & 30 & 31 & 33 & 34 \\ \hline
\ell_N & 24 &  8 & 12 &  4 &  2 & 12 &  4 &  4 & 12 &  2 &  2 &  4 & 12 &  6 &  2 &  4 &  4 &  4 &  1 & 12 &  2 &  4 \\
k_N    & 12 &  8 & 12 &  4 &  4 & 12 &  8 &  4 & 12 &  4 &  4 &  4 & 12 & 12 &  4 &  4 &  8 &  4 &  4 & 12 &  4 &  8 \\
\end{array}
$$

$$
\begin{array}{l|cccccccccccccccccccccc}
N      & 35 & 38 & 39 & 41 & 42 & 46 & 47 & 51 & 55 & 59 & 62 & 66 & 69 & 70 & 71 & 78 & 87 & 94 & 95 &105 &110 &119 \\ \hline
\ell_N &  2 &  2 &  6 &  4 &  1 &  2 &  4 &  2 &  2 &  4 &  2 &  1 &  2 &  1 &  4 &  1 &  2 &  2 &  2 &  1 &  1 &  2 \\
k_N    &  4 &  4 & 12 &  4 &  4 &  4 &  4 &  4 &  4 &  4 &  4 &  4 &  4 &  4 &  4 &  4 &  4 &  4 &  4 &  4 &  4 &  4 \\
\end{array}
$$
\end{table}

\begin{table}
\caption{\label{tab:1}The degree $\ell_N$ and the weight $k_N$ of the modular form $\Delta_N$ on
$\Gamma_0(N)^+$ for all moonshine groups of genus one.
Listed are the level $N$ and the values of degree $\ell_N$ and weight $k_N$.}
$$
\begin{array}{l|ccccccccccccccccccc}
N      & 37 & 43 & 53 & 57 & 58 & 61 & 65 & 74 & 77 & 79 & 82 & 83 & 86 & 89 & 91 &101 &102 &111 &114 \\ \hline
\ell_N & 12 & 12 &  4 &  6 &  4 & 12 &  2 &  4 &  2 & 12 &  4 &  4 &  2 &  4 &  6 &  4 &  1 &  6 &  1 \\
k_N    & 12 & 12 &  4 & 12 &  8 & 12 &  4 &  8 &  4 & 12 &  8 &  4 &  4 &  4 & 12 &  4 &  4 & 12 &  4 \\
\end{array}
$$

$$
\begin{array}{l|ccccccccccccccccccc}
N      &118 &123 &130 &131 &138 &141 &142 &143 &145 &155 &159 &174 &182 &190 &195 &210 &222 &231 &238 \\ \hline
\ell_N &  2 &  2 &  2 &  4 &  1 &  2 &  2 &  2 &  2 &  2 &  2 &  1 &  1 &  1 &  1 &  1 &  1 &  1 &  1 \\
k_N    &  4 &  4 &  8 &  4 &  4 &  4 &  4 &  4 &  4 &  4 &  4 &  4 &  4 &  4 &  4 &  8 &  4 &  4 &  4 \\
\end{array}
$$
\end{table}

\section{Holomorphic Eisenstein series on $\Gamma_0(N)^+$}\label{holomorphic Eisenstein}

In the classical $\Gamma_0(1)=\PSL(2,\Z)$ case, holomorphic Eisenstein series of even weight
$k\geq 4$ are defined by formula \eqref{E k defin}.  Modularity is an immediate consequence of
the definition.  We proceed analogously in the moonshine case and define for even $k\geq 4$
modular forms by \eqref{E k, p defin}.  The following proposition relates $E_k^{(N)}(z)$
to $E_k(z)$.

\begin{proposition}
Let $E_k^{(N)}(z)$ be the modular form defined by \eqref{E k, p defin}.
Then, for all even $k=2m \geq 4$ and all $z\in \h$ one has
\begin{align}\label{E_k, p proposit fla}
E_{2m}^{(N)}(z)= \frac1{\sigma_m(N)} \sum_{v \mid N}v^m E_{2m}(vz),
\end{align}
where $E_k$ is defined by \eqref{E k defin}.
\end{proposition}

\begin{proof}
Taking $e\in\{v : v \mid N\}$ in the definition \eqref{defn moons group}
of the moonshine group we see that
$$
E_k^{(N)}(z)
= \sum_{v \mid N} \frac12 v^{-k/2}
\sum_{(c,d)\in\Z^2, (c,v)=1, ((N/v)c,d)=1} \left(\frac{N}{v}cz+d \right)^{-k}
= \sum_{v \mid N} v^{-k/2} E_k^{(N,v)}(z),
$$
where
$$
E_k^{(N,v)}(z)=\frac12\sum_{(c,d)\in\Z^2, ((N/v)c,vd)=1}
\left(\frac{N}{v}cz+d \right)^{-k}.
$$

Now, we use the fact that $N$ is square-free to deduce that for a fixed $v_1 \mid N$ we have
\begin{multline*}
\{(x,y)\in\Z^2 : (x,y)=1\} = \biguplus_{u_1 \mid v_1,\, u_2 \mid (N/v_1)} \{(x,y)\in\Z^2 :
(x,y)=1,\ (x,v_1)=u_1,\ (y,(N/v_1))=u_2\} \\
= \biguplus_{u_1 \mid v_1,\, u_2 \mid (N/v_1)} \{(x,y)\in\Z^2 :
x=u_1x_1,\ y=u_2y_1,\ (\frac{N}{v_1 u_2}u_1 x_1, \frac{v_1}{u_1} u_2 y_1)=1,
\ (x_1,y_1)\in\Z^2 \}.
\end{multline*}
Therefore, for a fixed $v_1 \mid N$ we have that
\begin{multline*}
E_k((N/v_1)z) = \frac12
\sum_{(c,d)\in\Z^2, (c,d)=1} \left(\frac{N}{v_1}cz+d \right)^{-k}
= \frac12 \sum_{u_1 \mid v_1,\, u_2 \mid (N/v_1)}
\sum_{(\frac{N}{v_1 u_2}u_1 c_1, \frac{v_1}{u_1} u_2 d_1)=1}
\left(\frac{N}{v_1}u_1c_1z+u_2d_1 \right)^{-k} \\
= \frac12 \sum_{u_1 \mid v_1,\, u_2 \mid (N/v_1)} u_2^{-k}
\sum_{(\frac{N}{v_1 u_2}u_1 c_1, \frac{v_1}{u_1} u_2 d_1)=1}
\left(\frac{N}{v_1 u_2}u_1c_1z+d_1 \right)^{-k}
= \sum_{u_1 \mid v_1,\, u_2 \mid (N/v_1)} u_2^{-k}
E_k^{(N, \frac{v_1 u_2}{u_1})}(z).
\end{multline*}
Multiplying the above equation by $v_1^{-k/2}$ and taking the sum over
all $v_1\mid N$ we get
$$
\sum_{v_1 \mid N} v_1^{-k/2} E_k((N/v_1)z)
= \sum_{v_1 \mid N} v_1^{-k/2} \sum_{u_1 \mid v_1,\, u_2 \mid (N/v_1)} u_2^{-k}
E_k^{(N, \frac{v_1 u_2}{u_1})}(z).
$$
Once we fix $v_1 \mid N$ and put $v=\frac{v_1 u_2}{u_1}$, it is easy to
see that $v$ ranges over all divisors of $N$, as $u_1$ runs through divisors of $v_1$
and $u_2$ runs through divisors of $N/v_1$.

Furthermore, for fixed $v_1, v \mid N$ there exists a unique pair
$(u_1, u_2)$ of positive integers such that
$u_1 \mid v_1,\, u_2 \mid (N/v_1)$ and $v=\frac{v_1 u_2}{u_1}$.
Reasoning in this way, we see that
$$
\sum_{v_1 \mid N} v_1^{-k/2} E_k((N/v_1)z)
= \sum_{v \mid N} v^{-k/2} E_k^{(N, v)}(z) \cdot \sum_{v_2 \mid N} v_2^{-k/2}
= \sigma_{-k/2}(N) E_k^{(N)}(z).
$$
By the definition of the function $\sigma_m(N)$ the above equation is
equivalent to \eqref{E_k, p proposit fla}, which completes the proof.
\end{proof}

\section{Searching for generators of function fields}\label{Generators of function fields}

With the above analysis, we are now able to define holomorphic modular functions on
the space $X_N= \overline{\Gamma_0(N)^+}\backslash\h$.

\subsection{$q$-expansions}

According to Theorem~\ref{Delta_N}, the form $\Delta_N$ has weight $k_N$, and its
$q$-expansion is easily obtained from inserting the product formula of the classical
$\Delta$ function
$$
\eta^{24}(z) = q\prod_{n=1}^\infty(1-q^n)^{24}, \quad q=e(z),\ z\in\h,
$$
in
$$
\Delta_N(z) = \left(\prod_{v \mid N} \eta(v z)\right)^{\ell_N}.
$$
This allows us to get the $q$-expansion for any power of $\Delta_N$, including negative powers.

Similarly, the $q$-expansion of the holomorphic Eisenstein series $E_k^{(N)}$ is obtained
from the $q$-expansion of the classical Eisenstein series $E_k$.  Using the notation of
\cite{Zag08}, we write
$$
E_k(z)=1-\frac{2k}{B_k}\sum_{\nu=1}^\infty \sigma_{k-1}(\nu)q^\nu,
$$
where $B_k$ stands for the $k$-th Bernoulli number (e.g.\ $-8/B_4=240$; $-12/B_6= -504$, etc.).
From \eqref{E_k, p proposit fla} we deduce
$$
E_k^{(N)}(z)=\frac1{\sigma_{k/2}(N)}\sum_{v \mid N}v^{k/2}\left(1-\frac{2k}{B_k}
\sum_{\nu=1}^\infty \sigma_{k-1}(\nu) q^{v\nu} \right); \quad q=e(z),\ z\in\h.
$$

\subsection{Subspaces of rational functions}

Now it is evident that for any positive integer $M$, the function\\[2ex]
\eqref{rational_function_b} \hspace*{\fill} $\displaystyle
F_b(z) = \prod_\nu\left(E_{m_\nu}^{(N)}(z)\right)^{b_\nu}\Big/\big(\Delta_N(z)\big)^M
\quad \text{where} \quad
\sum_\nu b_\nu m_\nu = Mk_N
\quad \text{and} \quad
b = (b_1, \ldots) $ \hspace*{\fill} \\[2ex]
is a holomorphic modular function on $\overline{\Gamma_0(N)^+}\backslash\h$, meaning a weight
zero modular form with polynomial growth near $i\infty$.  Its $q$-expansion follows
from substituting $E_k^{(N)}$ and $\Delta_N$ by their $q$-expansions.

Let $\s_M$ denote the set of all possible rational functions defined in
\eqref{rational_function_b} for all possible vectors $b=(b_{\nu})$ and $m=(m_{\nu})$ with fixed $M$.
Our rationale for finding a set of generators for the function field of the smooth,
compact algebraic curve associated to $X_N$ is based on the
assumption that a finite span of $\s = \cup_{M=0}^K \s_M$ contains the set of generators
for the function field.  Subject to this assumption, the set of generators follows
from a base change.  The base change is performed as follows.

\subsection{Our algorithm}\label{Algorithm}

Choose a non-negative integer $\kappa$.  Let $M=1$ and set $\s = \s_1 \cup \s_0$.
\begin{enumerate}
\item Form the matrix $A_{\s}$ of coefficients from the $q$-expansions of all elements of $\s$,
where each element in $\s$ is expanded along a row with each column containing the coefficient
of a power, negative, zero or positive, of $q$.  The expansion is recorded out to order
$q^{\kappa}$.
\item Apply Gauss elimination to $A_{\s}$, thus producing a matrix $B_{\s}$ which is in row-reduced
echelon form.
\item Implement the following decision to determine if the algorithm is complete:
If the $g$ lowest non-trivial rows at the bottom of $B_{\s}$ correspond to $q$-expansions whose
lead terms have precisely $g$ gaps in the set $\{q^{-1}, \ldots, q^{-2g}\}$, then the algorithm is
completed.  If the indicator to stop fails, then replace $M$ by $M+1$, $\s$ by $\s_M \cup \s$
and reiterate the algorithm.
\end{enumerate}

We described the algorithm with choice of an arbitrary $\kappa$.  For reasons of efficiency, we
initially selected $\kappa$ to be zero, so that all coefficients for $q^\nu$
for $\nu \leq \kappa$ are included in $A_{\s}$, but finally increased $\kappa$ to its desired value
according to Remark~\ref{rem:kappa algorithm}.

The rationale for the stopping decision in Step 3 above is based on two ideas, one factual
and one hopeful.  First, the Weierstrass gap theorem states that for any point $P$ on a
compact Riemann surface there are precisely $g$ gaps in the set of possible orders from $1$
to $2g$ of functions whose only pole is at $P$.  For the main considerations of this paper,
we study $g\leq3$.  For instance, when $g=1$ there occurs exactly one gap which for topological
reasons is always $\{q^{-1}\}$.  Second, for any genus, the assumption which is hopeful is
that the function field is generated by the set of holomorphic modular functions defined in
\eqref{rational_function_b}, which is related to the question of whether the ring of holomorphic
modular forms on $\Gamma_0(N)^+$ is generated by holomorphic Eisenstein series and the
Kronecker limit function.  The latter assumption is not obvious, but it has turned out to be true
for all moonshine groups $\Gamma_0(N)^+$ that we have studied so far.  This includes
in particular all moonshine groups of genus zero, genus one, genus two, and genus three.

\subsection{Implementation notes}

We have implemented C code that generates in integer arithmetic the set
$\s=\cup_{M=0}^k \s_M$ of all possible rational functions defined in
\eqref{rational_function_b} for any given positive integer $K$ and
any square-free level $N$.  The set $\s$ is passed in symbolic notation to
PARI/GP \cite{PARI11} which computes in rational arithmetic the $q$-expansions
of all elements of $\s$ up to any given positive order $\kappa$ and forms the
matrix $A_{\s}$.  Theoretically, we could have used PARI/GP for the Gauss
elimination of $A_{\s}$.  However, for reasons of efficiency, we have written
our own C code, linked against the GMP MP library \cite{Gra12}, to compute the
Gauss elimination in rational arithmetic.  The result is a matrix $B_{\s}$ whose
rows correspond to $q$-expansions with rational coefficients.  Inspecting the
rational coefficients of $B_{\s}$, we find that the denominators are $1$,
i.e.\ after Gauss elimination the coefficients turn out to be integers.
All computations are in rational arithmetic, hence are exact.

After the fact, we select $\kappa$ as in Tables~\ref{tab:kappa genus zero} and
\ref{tab:kappa genus one} in case when genus is zero or one, or compute $\kappa$ as explained
in Remark~\ref{rem:kappa algorithm} when genus is bigger than one, and extend the $q$-expansions
of the field generators out to order $q^\kappa$.  Inspecting the coefficients we find that they
are integer which proves that the field generators have integer $q$-expansions.

It is important that we compute the Gauss elimination with a small value of $\kappa$
first, and extend the $q$-expansions for the field generators out to the
larger number $\kappa$ according to Remark~\ref{rem:kappa algorithm}, later.
There are two reasons for this: (i) It would be a waste of resources to
start with a large value of $\kappa$.  Once the Gauss elimination is done
with a small value of $\kappa$, it is easy to extend the $q$-expansions
out to a larger value of $\kappa$ using reverse engineering of the Gauss
elimination.  (ii) If the genus is larger than one, we do not know the gap
sequence for the field generators in advance.  We can determine $a_1$ and $a_2$, and hence
$\kappa$ according to Remark~\ref{rem:kappa algorithm} only if we have completed the Gauss
elimination with some other positive value of $\kappa$ before.

\section{Genus zero}\label{genus zero}

The genus zero case, due it simplicity is suitable for demonstration of the algorithm.  We do
that at the example of holomorphic modular functions on $X_N$ for the level $N=17$.  After four
iterations, the set $\s$ consists of $15$ functions.  With the above notation, the functions
are as follows, with their associated $q$-expansions:
$$
\begin{array}{rcl}
F_{(0)}(z) & = & \ \ 1 \\
F_{(4)}(z) & = & \ q^{-3} + 140/29 \cdot q^{-2} + 718/29 \cdot q^{-1} + \ldots \\
F_{(4,4)}(z) & = & \ q^{-6} + 280/29 \cdot q^{-5} + 61244/841 \cdot q^{-4} + \ldots \\
F_{(8)}(z) & = & \ q^{-6} + 334328/41761 \cdot q^{-5} + 1870364/41761 \cdot q^{-4} + \ldots \\
F_{(4,4,4)}(z) & = & \ q^{-9} + 420/29 \cdot q^{-8} + 121266/841 \cdot q^{-7} + \ldots \\
F_{(6,6)}(z) & = & \ q^{-9} + 460/39 \cdot q^{-8} + 122866/1521 \cdot q^{-7} + \ldots \\
F_{(8,4)}(z) & = & \ q^{-9} + 15542052/1211069 \cdot q^{-8} + 4518306/41761 \cdot q^{-7} + \ldots \\
F_{(12)}(z) & = & \ q^{-9} + 20014879596/1667906087 \cdot q^{-8} + 150125051502/1667906087 \cdot q^{-7} + \ldots \\
F_{(4,4,4,4)}(z) & = & q^{-12} + 560/29 \cdot q^{-11} + 200888/841 \cdot q^{-10} + \ldots \\
F_{(6,6,4)}(z) & = & q^{-12} + 18800/1131 \cdot q^{-11} + 7166792/44109 \cdot q^{-10} + \ldots \\
F_{(8,4,4)}(z) & = & q^{-12} + 21388592/1211069 \cdot q^{-11} + 6845330168/35121001 \cdot q^{-10} + \ldots \\
F_{(8,8)}(z) & = & q^{-12} + 668656/41761 \cdot q^{-11} + 267991753592/1743981121 \cdot q^{-10} + \ldots \\
F_{(10,6)}(z) & = & q^{-12} + 4445968/279669 \cdot q^{-11} + 123230008/839007 \cdot q^{-10} + \ldots \\
F_{(12,4)}(z) & = & q^{-12} + 28066840016/1667906087 \cdot q^{-11} + 8353266207464/48369276523 \cdot q^{-10} + \ldots \\
F_{(16)}(z) & = & q^{-12} + 201850517349872/12615657333857 \cdot q^{-11} + 1917580182271864/12615657333857 \cdot q^{-10} + \ldots \\
\end{array}
$$
After Gauss elimination, we derive a set of functions with the following $q$-expansions:
$$
\begin{array}{rcrcl}
  f_{1}(z) & = & q^{-12} & + & 23952 q + 1146120 q^{2} + 24348512 q^{3} + 333173679 q^{4} + 3413935488 q^{5} + 28384768032 q^{6} + \ldots \\
  f_{2}(z) & = & q^{-11} & + & 15037 q + 592636 q^{2} + 10928808 q^{3} + 132565620 q^{4} + 1221797082 q^{5} + 9229730400 q^{6} + \ldots \\
  f_{3}(z) & = & q^{-10} & + & 9040 q + 297580 q^{2} + 4718080 q^{3} + 50477640 q^{4} + 416335968 q^{5} + 2844975860 q^{6} + \ldots \\
  f_{4}(z) & = & q^{-9} & + & 5427 q + 143478 q^{2} + 1949433 q^{3} + 18261234 q^{4} + 134057376 q^{5} + 824326644 q^{6} + \ldots \\
  f_{5}(z) & = & q^{-8} & + & 3088 q + 66520 q^{2} + 763488 q^{3} + 6223924 q^{4} + 40378496 q^{5} + 222115728 q^{6} + \ldots \\
  f_{6}(z) & = & q^{-7} & + & 1722 q + 29120 q^{2} + 281288 q^{3} + 1972936 q^{4} + 11231038 q^{5} + 54855304 q^{6} + \ldots \\
  f_{7}(z) & = & q^{-6} & + & 888 q + 12063 q^{2} + 95680 q^{3} + 573060 q^{4} + 2830848 q^{5} + 12176072 q^{6} + \ldots \\
  f_{8}(z) & = & q^{-5} & + & 460 q + 4520 q^{2} + 29620 q^{3} + 148560 q^{4} + 630492 q^{5} + 2359040 q^{6} + \ldots \\
  f_{9}(z) & = & q^{-4} & + & 200 q + 1572 q^{2} + 7984 q^{3} + 33267 q^{4} + 118848 q^{5} + 382040 q^{6} + \ldots \\
  f_{10}(z) & = & q^{-3} & + & 87 q + 444 q^{2} + 1816 q^{3} + 5988 q^{4} + 17772 q^{5} + 47840 q^{6} + \ldots \\
  f_{11}(z) & = & q^{-2} & + & 28 q + 107 q^{2} + 296 q^{3} + 786 q^{4} + 1808 q^{5} + 4021 q^{6} + \ldots \\
  f_{12}(z) & = & q^{-1} & + & 7 q + 14 q^{2} + 29 q^{3} + 50 q^{4} + 92 q^{5} + 148 q^{6} + 246 q^{7} + 386 q^{8} + 603 q^{9} + 904 q^{10} + 1367 q^{11} + \ldots \\
  f_{13}(z) & = & & 1 \\
  f_{14}(z) & = & & 0 & + \quad O(q^{\kappa+1}) \\
  f_{15}(z) & = & & 0 & + \quad O(q^{\kappa+1}) \\
\end{array}
$$
Every function $f_k$ with $k=1,\ldots,11$ can be written as a polynomial in $f_{12}$.  For
example, $f_{11} = f_{12}^2 - 14$.  Further relations are easily derived.

As stated, the $j$-invariant, also called the Hauptmodul, is the unique holomorphic modular
function with simple pole at $z=i\infty$.  For $\Gamma_0(17)^+$, the Hauptmodul is given by
$j_{17}=f_{12}(z)$.  It is remarkable that the $q$-expansion for $j_{17}$ has integer
coefficients, especially when considering how we obtained $j_{17}$.

In the case when the surface $X_N$ has genus equal to zero, $q$-expansions of Hauptmoduls $j_N$
coincide with McKay-Thompson series for certain classes of the monster $\M$.  Therefore, we were
able to compare results obtained using the algorithm explained above with well known expansions
that may be found e.g.\ at \cite{Sloane2010}.  All expansions were computed out to order
$q^{\kappa}$ and they match with the corresponding expansions of McKay-Thompson series
available at \cite{Sloane2010}.

The successful execution of our algorithm in the genus zero case, besides direct check of the
correctness of the algorithm provides two additional corollaries.  First, the $j$-invariant is
expressed as a linear combination of elements from the set $\s_M$, for some $M$, thus showing
that the Hauptmoduli can be written as a rational function in holomorphic Eisenstein series and
the Kronecker limit function, thus generalizing \eqref{j_via_Eisenstein}.  Second, since the
function field associated to the underlying algebraic curve is generated by the $j$-invariant,
we conclude that all holomorphic modular forms on $X_N$ are generated by a finite set of
holomorphic Eisenstein series and a power of the Kronecker limit function, again generalizing
a classical result for $\PSL(2,\Z)$.  A complete analysis of these two results for all genus
zero moonshine groups will be given in \cite{JST13a}.

\section{Genus one}\label{genus one}

The smallest $N$ such that $\Gamma_0(N)^+$ has genus one is $N=37$.  In that example, the
Kronecker limit function is $\Delta_{37}(z)=\sqrt{\Delta(z)\Delta(37z)}$, which has weight $12$.
After four iterations, the algorithm completes successfully.  However, $\s_4$ consists of
$434$ functions, and this set contains functions whose pole at $i\infty$ has order as large as
$76$.  For instance, the function $F_{(48)}(z)=E^{(37)}_{48}(z)/\Delta_{37}^4(z)$ is in $\s_4$,
and the first few terms in the $q$-expansion of $F_{(48)}$ are
$$
F_{(48)}(z) = q^{-76} + a_{-75}/d \cdot q^{-75} + a_{-74}/d \cdot q^{-74} + a_{-73}/d \cdot q^{-73}+ a_{-72}/d \cdot q^{-72} + a_{-71}/d \cdot q^{-71} + a_{-70}/d \cdot q^{-70} + \ldots
$$
where $a_{-75},\ldots,a_{-70}$, and $d$ are the following integers:
\begin{align*}
a_{-75}&=343178734921291171422482658701977036706530252314069069522767928048;\\
a_{-74}&=8751057740492924871273307796900414436016521434027203513304667425784;\\
a_{-73}&=156947074770670495730548735913037498120456652924086413129066082281152;\\
a_{-72}&=2215476117968125479910692423915288257313864521881655533695941651024796;\\
a_{-71}&=26148160528592859515364643697138531568926425697820573801334173510055392;\\
a_{-70}&=267886235622853011855513756670347976534161565268720843078120874364187648;
\end{align*}
and
$$
d= 7149556977526899404635055389624521598052713589876438948390995771.
$$

In the setting of genus one, there is no holomorphic modular function with a simple pole at
$i\infty$.  However, there does exist a holomorphic modular function with an order two pole
at $i\infty$, which is equal to the Weierstrass $\wp$-function.  We let $j_{1;N}$ denote the
$\wp$-function for $\Gamma_0(N)^+$.  Also, let $j_{2;N}$ denote the holomorphic modular
function with a third order pole at $i\infty$.  In the coordinates of the complex plane,
$j_{2;N}$ is the derivative of $j_{1;N}$.  However, in the coordinate of the upper half plane,
either $z$ or $q$, one will not have that $j_{2;N}$ is the derivative of $j_{1;N}$, as can be
seen by the chain rule from one variable calculus.  In all cases, the functions $j_{1;N}$ and
$j_{2;N}$ generate the function field of the underlying elliptic curve.

Our algorithm yields the $q$-expansions of the functions $j_{1;N}$ and $j_{2;N}$.  In the case
$N=37$, we obtain the expansions
$$
j_{2;37}(z) = q^{-3} + 3q^{-1} + 19q + 38q^{2} + 93q^{3} + 176q^{4} + 347q^{5} + 630q^{6} + 1139q^{7} + 1944q^{8} + 3305q^{9} + 5404q^{10} + \ldots
$$
and
$$
j_{1;37}(z) = q^{-2} + 2q^{-1} + 9q + 18q^{2} + 29q^{3} + 51q^{4} + 82q^{5} + 131q^{6} + 199q^{7} + 306q^{8} + 450q^{9} + 666q^{10} + 957q^{11} + \ldots
$$

As we have stated, the coefficients in the $q$-expansions of $j_{1;37}$ and $j_{2;37}$ are
integers.  Again, noting the size of the denominators of the coefficients in the $q$-expansions
of elements of $\s_4$, it is striking that the generators of the function fields would have
integer coefficients.

\begin{sidewaystable}
\caption{\label{tab:y & x genus one}$q$-expansion of $y = j_{2;N}(z)$, and $q$-expansion
of $x = j_{1;N}(z)$ for the genus one moonshine groups $\Gamma_0(N)^+$.}
$$
\begin{array}{ccll}
N  &&\hspace*{.18\textheight} y  &\hspace*{.18\textheight}  x \\[1ex]
37  &&  q^{-3} + 3q^{-1} + 19q + 38q^{2} + 93q^{3} + 176q^{4} + 347q^{5} + 630q^{6} + 1139q^7 + \ldots  &  q^{-2} + 2q^{-1} + 9q + 18q^{2} + 29q^{3} + 51q^{4} + 82q^{5} + 131q^{6} + 199q^{7} + \ldots \\
43  &&  q^{-3} + 2q^{-1} + 13q + 24q^{2} + 55q^{3} + 98q^{4} + 186q^{5} + 318q^{6} + 549q^{7} + \ldots  &  q^{-2} + 2q^{-1} + 7q + 13q^{2} + 20q^{3} + 33q^{4} + 50q^{5} + 77q^{6} + 112q^{7} + \ldots \\
53  &&  q^{-3} + 3q^{-1} + 10q + 16q^{2} + 33q^{3} + 50q^{4} + 90q^{5} + 140q^{6} + 227q^{7} + \ldots  &  q^{-2} + q^{-1} + 4q + 7q^{2} + 10q^{3} + 17q^{4} + 23q^{5} + 35q^{6} + 48q^{7} + \ldots \\
57  &&  q^{-3} + q^{-1} + 6q + 10q^{2} + 27q^{3} + 36q^{4} + 61q^{5} + 106q^{6} + 156q^{7} + \ldots  &  q^{-2} + 2q^{-1} + 5q + 8q^{2} + 9q^{3} + 15q^{4} + 23q^{5} + 30q^{6} + 43q^{7} + \ldots \\
58  &&  q^{-3} + 3q^{-1} + 10q + 10q^{2} + 30q^{3} + 34q^{4} + 72q^{5} + 90q^{6} + 172q^{7} + \ldots  &  q^{-2} + q^{-1} + 3q + 7q^{2} + 7q^{3} + 14q^{4} + 17q^{5} + 29q^{6} + 32q^{7} + \ldots \\
61  &&  q^{-3} + 2q^{-1} + 7q + 10q^{2} + 22q^{3} + 32q^{4} + 53q^{5} + 80q^{6} + 127q^{7} + \ldots  &  q^{-2} + q^{-1} + 3q + 6q^{2} + 7q^{3} + 11q^{4} + 16q^{5} + 23q^{6} + 30q^{7} + \ldots \\
65  &&  q^{-3} + 2q^{-1} + 7q + 10q^{2} + 17q^{3} + 24q^{4} + 43q^{5} + 66q^{6} + 102q^{7} + \ldots  &  q^{-2} + q^{-1} + 3q + 4q^{2} + 7q^{3} + 11q^{4} + 13q^{5} + 19q^{6} + 23q^{7} + \ldots \\
74  &&  q^{-3} + 3q^{-1} + 7q + 6q^{2} + 17q^{3} + 16q^{4} + 35q^{5} + 38q^{6} + 71q^{7} + \ldots  &  q^{-2} \quad + \ \ \quad q + 4q^{2} + 3q^{3} + 7q^{4} + 6q^{5} + 13q^{6} + 13q^{7} + 22q^{8} + \ldots \\
77  &&  q^{-3} + 3q^{-1} + 7q + 6q^{2} + 14q^{3} + 16q^{4} + 28q^{5} + 38q^{6} + 56q^{7} + \ldots  &  q^{-2} \quad + \ \ \quad q + 3q^{2} + 3q^{3} + 7q^{4} + 7q^{5} + 10q^{6} + 13q^{7} + 16q^{8} + \ldots \\
79  &&  q^{-3} + q^{-1} + 4q + 5q^{2} + 11q^{3} + 14q^{4} + 23q^{5} + 32q^{6} + 47q^{7} + \ldots  &  q^{-2} + q^{-1} + 2q + 4q^{2} + 4q^{3} + 6q^{4} + 8q^{5} + 11q^{6} + 14q^{7} + 19q^{8} + \ldots \\
82  &&  q^{-3} + 2q^{-1} + 5q + 4q^{2} + 13q^{3} + 12q^{4} + 24q^{5} + 26q^{6} + 49q^{7} + \ldots  &  q^{-2} + q^{-1} + 2q + 4q^{2} + 3q^{3} + 6q^{4} + 7q^{5} + 11q^{6} + 11q^{7} + 18q^{8} + \ldots \\
83  &&  q^{-3} + q^{-1} + 4q + 5q^{2} + 9q^{3} + 12q^{4} + 20q^{5} + 27q^{6} + 40q^{7} + \ldots  &  q^{-2} + q^{-1} + 2q + 3q^{2} + 4q^{3} + 6q^{4} + 7q^{5} + 10q^{6} + 12q^{7} + 16q^{8} + \ldots \\
86  &&  q^{-3} + 2q^{-1} + 5q + 4q^{2} + 11q^{3} + 10q^{4} + 22q^{5} + 22q^{6} + 41q^{7} + \ldots  &  q^{-2} \quad + \ \ \quad q + 3q^{2} + 2q^{3} + 5q^{4} + 4q^{5} + 9q^{6} + 8q^{7} + 14q^{8} + \ldots \\
89  &&  q^{-3} + q^{-1} + 3q + 4q^{2} + 9q^{3} + 10q^{4} + 16q^{5} + 22q^{6} + 32q^{7} + \ldots  &  q^{-2} + q^{-1} + 2q + 3q^{2} + 3q^{3} + 5q^{4} + 6q^{5} + 8q^{6} + 10q^{7} + 13q^{8} + \ldots \\
91  &&  q^{-3} \quad + \ \ \quad 2q + 4q^{2} + 6q^{3} + 10q^{4} + 15q^{5} + 18q^{6} + 28q^{7} + \ldots  &  q^{-2} + 2q^{-1} + 3q + 3q^{2} + 4q^{3} + 5q^{4} + 7q^{5} + 10q^{6} + 11q^{7} + \ldots \\
101  &&  q^{-3} + 2q^{-1} + 3q + 4q^{2} + 7q^{3} + 8q^{4} + 13q^{5} + 16q^{6} + 23q^{7} + \ldots  &  q^{-2} \quad + \ \ \quad q + 2q^{2} + 2q^{3} + 3q^{4} + 3q^{5} + 5q^{6} + 6q^{7} + 8q^{8} + \ldots \\
102  &&  q^{-3} + q^{-1} + 3q + 2q^{2} + 8q^{3} + 6q^{4} + 12q^{5} + 14q^{6} + 22q^{7} + \ldots  &  q^{-2} + q^{-1} + q + 3q^{2} + 2q^{3} + 4q^{4} + 5q^{5} + 6q^{6} + 6q^{7} + 10q^{8} + \ldots \\
111  &&  q^{-3} \quad + \ \ \quad q + 2q^{2} + 6q^{3} + 5q^{4} + 8q^{5} + 12q^{6} + 14q^{7} + 18q^{8} + \ldots  &  q^{-2} + 2q^{-1} + 3q + 3q^{2} + 2q^{3} + 3q^{4} + 4q^{5} + 5q^{6} + 7q^{7} + 9q^{8} + \ldots \\
114  &&  q^{-3} + q^{-1} + 2q + 2q^{2} + 7q^{3} + 4q^{4} + 9q^{5} + 10q^{6} + 16q^{7} + 16q^{8} + \ldots  &  q^{-2} \quad + \ \ \quad q + 2q^{2} + q^{3} + 3q^{4} + q^{5} + 4q^{6} + 5q^{7} + 6q^{8} + \ldots \\
118  &&  q^{-3} + q^{-1} + 3q + 2q^{2} + 5q^{3} + 4q^{4} + 9q^{5} + 8q^{6} + 15q^{7} + 14q^{8} + \ldots  &  q^{-2} + q^{-1} + q + 2q^{2} + 2q^{3} + 3q^{4} + 3q^{5} + 5q^{6} + 4q^{7} + 7q^{8} + \ldots \\
123  &&  q^{-3} + q^{-1} + 2q + 2q^{2} + 5q^{3} + 4q^{4} + 7q^{5} + 10q^{6} + 11q^{7} + 14q^{8} + \ldots  &  q^{-2} \quad + \ \ \quad q + q^{2} + q^{3} + 3q^{4} + 2q^{5} + 3q^{6} + 4q^{7} + 4q^{8} + \ldots \\
130  &&  q^{-3} + 2q^{-1} + 3q + 2q^{2} + 5q^{3} + 4q^{4} + 7q^{5} + 6q^{6} + 14q^{7} + 10q^{8} + \ldots  &  q^{-2} - q^{-1} - q + 2q^{2} + q^{3} + q^{4} + q^{5} + 3q^{6} + q^{7} + 5q^{8} + 2q^{9} + \ldots \\
131  &&  q^{-3} + q^{-1} + 2q + 2q^{2} + 4q^{3} + 4q^{4} + 6q^{5} + 7q^{6} + 10q^{7} + 12q^{8} + \ldots  &  q^{-2} \quad + \ \ \quad q + q^{2} + q^{3} + 2q^{4} + 2q^{5} + 3q^{6} + 3q^{7} + 4q^{8} + \ldots \\
138  &&  q^{-3} + q^{-1} + 2q + q^{2} + 5q^{3} + 3q^{4} + 5q^{5} + 6q^{6} + 10q^{7} + 8q^{8} + \ldots  &  q^{-2} + q^{-1} + q + 2q^{2} + q^{3} + 2q^{4} + 3q^{5} + 3q^{6} + 2q^{7} + 5q^{8} + \ldots \\
141  &&  q^{-3} + q^{-1} + q + 2q^{2} + 4q^{3} + 3q^{4} + 5q^{5} + 7q^{6} + 8q^{7} + 9q^{8} + \ldots  &  q^{-2} \quad + \ \ \quad q + q^{2} + q^{3} + 2q^{4} + q^{5} + 2q^{6} + 3q^{7} + 3q^{8} + \ldots \\
142  &&  q^{-3} + q^{-1} + 2q + q^{2} + 4q^{3} + 3q^{4} + 6q^{5} + 5q^{6} + 9q^{7} + 8q^{8} + \ldots  &  q^{-2} + q^{-1} + q + 2q^{2} + q^{3} + 2q^{4} + 2q^{5} + 3q^{6} + 3q^{7} + 4q^{8} + \ldots \\
143  &&  q^{-3} + q^{-1} + q + 2q^{2} + 3q^{3} + 3q^{4} + 6q^{5} + 7q^{6} + 8q^{7} + 9q^{8} + \ldots  &  q^{-2} \quad + \ \ \quad q + q^{2} + q^{3} + 2q^{4} + q^{5} + 2q^{6} + 2q^{7} + 3q^{8} + \ldots \\
145  &&  q^{-3} \quad + \ \ \quad q + q^{2} + 4q^{3} + 2q^{4} + 4q^{5} + 6q^{6} + 7q^{7} + 9q^{8} + \ldots  &  q^{-2} + q^{-1} + q + 2q^{2} + q^{3} + 2q^{4} + 2q^{5} + 2q^{6} + 3q^{7} + 4q^{8} + \ldots \\
155  &&  q^{-3} + q^{-1} + q + 2q^{2} + 3q^{3} + 3q^{4} + 4q^{5} + 4q^{6} + 7q^{7} + 7q^{8} + \ldots  &  q^{-2} \quad + \ \ \quad q + q^{2} + q^{3} + q^{4} + q^{5} + 2q^{6} + 2q^{7} + 3q^{8} + 2q^{9} + \ldots \\
159  &&  q^{-3} \quad + \ \ \quad q + q^{2} + 3q^{3} + 2q^{4} + 3q^{5} + 5q^{6} + 5q^{7} + 6q^{8} + \ldots  &  q^{-2} + q^{-1} + q + q^{2} + q^{3} + 2q^{4} + 2q^{5} + 2q^{6} + 3q^{7} + 3q^{8} + \ldots \\
174  &&  q^{-3} \quad + \ \ \quad q + q^{2} + 3q^{3} + q^{4} + 3q^{5} + 3q^{6} + 4q^{7} + 4q^{8} + \ldots  &  q^{-2} + q^{-1} + q^{2} + q^{3} + 2q^{4} + 2q^{5} + 2q^{6} + 2q^{7} + 3q^{8} + 2q^{9} + \ldots \\
182  &&  q^{-3} \quad + \ \ \quad 2q + 2q^{3} + 2q^{4} + 3q^{5} + 2q^{6} + 4q^{7} + 4q^{8} + 7q^{9} + \ldots  &  q^{-2} \quad + \ \ \quad q + q^{2} + q^{4} + q^{5} + 2q^{6} + q^{7} + 2q^{8} + q^{9} + 3q^{10} + \ldots \\
190  &&  q^{-3} + q^{-1} + 2q + q^{2} + 2q^{3} + q^{4} + 3q^{5} + 3q^{6} + 5q^{7} + 3q^{8} + \ldots  &  q^{-2} + q^{-1} + q + q^{2} + q^{3} + q^{4} + q^{5} + 2q^{6} + q^{7} + 3q^{8} + 2q^{9} + \ldots \\
195  &&  q^{-3} - q^{-1} + q + q^{2} + 2q^{3} + q^{5} + 3q^{6} + 3q^{7} + 3q^{8} + 5q^{9} + \ldots  &  q^{-2} + q^{-1} + q^{2} + q^{3} + 2q^{4} + q^{5} + q^{6} + 2q^{7} + 2q^{8} + 2q^{9} + \ldots \\
210  &&  q^{-3} + q^{-1} + q + 3q^{3} + 2q^{4} + 2q^{5} + 2q^{6} + 3q^{7} + 2q^{8} + 6q^{9} + \ldots  &  q^{-2} + q^{-1} + q + 2q^{2} + q^{4} + q^{5} + q^{6} + q^{7} + q^{8} + 2q^{9} + 2q^{10} + \ldots \\
222  &&  q^{-3} \quad + \ \ \quad q + 2q^{3} + q^{4} + 2q^{5} + 2q^{6} + 2q^{7} + 2q^{8} + 5q^{9} + \ldots  &  q^{-2} \quad + \ \ \quad q + q^{2} + q^{4} + q^{6} + q^{7} + q^{8} + q^{9} + 2q^{10} + q^{11} + \ldots \\
231  &&  q^{-3} \quad + \ \ \quad q + 2q^{3} + q^{4} + q^{5} + 2q^{6} + 2q^{7} + 3q^{8} + 4q^{9} + \ldots  &  q^{-2} \quad + \ \ \quad q + q^{4} + q^{5} + q^{6} + q^{7} + q^{8} + q^{9} + 2q^{10} + q^{11} + \ldots \\
238  &&  q^{-3} \quad + \ \ \quad q + q^{2} + q^{3} + q^{4} + 2q^{5} + q^{6} + 2q^{7} + q^{8} + 4q^{9} + \ldots  &  q^{-2} + q^{-1} + q^{3} + q^{4} + q^{5} + 2q^{6} + q^{7} + 2q^{8} + q^{9} + 2q^{10} + \ldots \\
\end{array}
$$
\end{sidewaystable}

\begin{table}
\caption{\label{pol:genus one}Cubic relation satisfied
by $x=j_{1;N}$ and $y=j_{2;N}$ for the genus one moonshine groups $\Gamma_0(N)^+$.}
$$
\begin{array}{rc}
N & \text{cubic relation} \\[1ex]
37 & y^2 - x^3 + 6xy - 6x^2 + 41y + 49x + 300 = 0 \\
43 & y^2 - x^3 + 6xy - 4x^2 + 33y + 43x + 192 = 0 \\
53 & y^2 - x^3 + 3xy - 6x^2 + 16y + x + 42 = 0 \\
57 & y^2 - x^3 + 6xy - 2x^2 + 25y + 37x + 114 = 0 \\
58 & y^2 - x^3 + 3xy - 6x^2 + 13y - 2x + 24 = 0 \\
61 & y^2 - x^3 + 3xy - 4x^2 + 12y + 7x + 30 = 0 \\
65 & y^2 - x^3 + 3xy - 4x^2 + 12y + x + 18 = 0 \\
74 & y^2 - x^3 - 6x^2 + 3y - 11x - 4 = 0 \\
77 & y^2 - x^3 - 6x^2 + 3y - 14x - 10 = 0 \\
79 & y^2 - x^3 + 3xy - 2x^2 + 8y + 8x + 15 = 0 \\
82 & y^2 - x^3 + 3xy - 4x^2 + 9y + 2x + 14 = 0 \\
83 & y^2 - x^3 + 3xy - 2x^2 + 8y + 5x + 12 = 0 \\
86 & y^2 - x^3 - 4x^2 + 3y - 5x = 0 \\
89 & y^2 - x^3 + 3xy - 2x^2 + 8y + 7x + 14 = 0 \\
91 & y^2 - x^3 + 6xy + 17y + 23x + 42 = 0 \\
101 & y^2 - x^3 - 4x^2 + 3y - 4x + 2 = 0 \\
102 & y^2 - x^3 + 3xy - 2x^2 + 5y + 4x + 6 = 0 \\
111 & y^2 - x^3 + 6xy + 17y + 25x + 48 = 0 \\
114 & y^2 - x^3 - 2x^2 + 3y + x + 2 = 0 \\
118 & y^2 - x^3 + 3xy - 2x^2 + 5y + x + 2 = 0 \\
123 & y^2 - x^3 - 2x^2 + 3y - 2x + 2 = 0 \\
130 & y^2 - x^3 - 3xy - 4x^2 - 6y - 3x + 2 = 0 \\
131 & y^2 - x^3 - 2x^2 + 3y - 2x + 1 = 0 \\
138 & y^2 - x^3 + 3xy - 2x^2 + 5y + 3x + 4 = 0 \\
141 & y^2 - x^3 - 2x^2 + 3y + 3 = 0 \\
142 & y^2 - x^3 + 3xy - 2x^2 + 5y + 3x + 6 = 0 \\
143 & y^2 - x^3 - 2x^2 + 3y + 5 = 0 \\
145 & y^2 - x^3 + 3xy + 4y + 7x + 4 = 0 \\
155 & y^2 - x^3 - 2x^2 + 3y + 2 = 0 \\
159 & y^2 - x^3 + 3xy + 4y + 4x + 3 = 0 \\
174 & y^2 - x^3 + 3xy + y + x = 0 \\
182 & y^2 - x^3 + 3y - x + 2 = 0 \\
190 & y^2 - x^3 + 3xy - 2x^2 + 5y = 0 \\
195 & y^2 - x^3 + 3xy + 2x^2 + x = 0 \\
210 & y^2 - x^3 + 3xy - 2x^2 + 5y + 5x + 10 = 0 \\
222 & y^2 - x^3 + 3y + x + 2 = 0 \\
231 & y^2 - x^3 + 3y - 2x + 2 = 0 \\
238 & y^2 - x^3 + 3xy + y - 2x - 2 = 0 \\
\end{array}
$$
\end{table}

In Table~\ref{tab:y & x genus one}, we list the generators of the function field of the
underlying algebraic curve for all $38$ genus one moonshine groups $\Gamma_0(N)^+$.  In
some instances, the computations behind Table~\ref{tab:y & x genus one} were quite intensive.
For $N=79$, the Kronecker limit function has weight $12$, and we needed seven iterations
of the algorithm.  When considering all functions whose numerator had weight up to $84$, the
set $\s_8$ had $13,158$ functions, and the orders of the pole at $i\infty$ was as large as $280$.
As $N$ became larger, the size of the denominators of the rational coefficients in the
$q$-expansions for $\s$ grew very large, in some cases becoming thousands of digits in length.
Nonetheless,  we have computed $q$-expansions of generators of the function field in all 38 genus
one cases out to order $q^{\kappa}$ where $\kappa$ is listed in Table~\ref{tab:kappa genus one}.
The exact arithmetic of the program and successful completion of the algorithm, together with
Theorem~\ref{genus_notzero_integrality} yields the following result.

\emph{For each genus one moonshine group, the two generators $j_{1;N}$ and $j_{2;N}$
of the function field have integer $q$-expansions.}

If we set $x=j_{1;37}$ and $y=j_{2;37}$, then it can be shown that $j_{1;37}$ and $j_{2;37}$ satisfy
the cubic relation
$$
y^2 - x^3 + 6xy - 6x^2 + 41y + 49x + 300 = 0.
$$
In Table~\ref{pol:genus one} we present the list of cubic relations satisfied by $j_{1;N}$ and
$j_{2;N}$ for all genus one moonshine groups.

As in the genus zero setting, there are two important consequences of the successful
implementation of our algorithm.  First, the functions $j_{1;N}$ and $j_{2;N}$ can be expressed
as linear combinations of elements in $\s_M$ for some $M$.  In particular, the Weierstrass
$\wp$-function $j_{1;N}$ can be written in terms of certain holomorphic Eisenstein series and
powers of the Kronecker limit function.  The same assertion holds for $j_{2;N}$.  Second, since
$j_{1;N}$ and $j_{2;N}$ generate the function field, we conclude that all holomorphic modular
forms can be written in terms of certain holomorphic Eisenstein series and the Kronecker limit
functions.  A full documentation of these results will be given in \cite{JST13b}.

\section{Higher genus examples}\label{higher genus}

\begin{sidewaystable}
\caption{\label{tab:y & x genus two}$q$-expansion of $y = j_{2;N}(z)$, and $q$-expansion
of $x = j_{1;N}(z)$, and $\kappa$ for the genus two moonshine groups $\Gamma_0(N)^+$.}
$$
\begin{array}{ccllc}
N   &&\hspace*{.18\textheight} y  &\hspace*{.18\textheight}  x  &  \kappa \\[1ex]
67  &&  q^{-4} + 3q^{-2} + 3q^{-1} + 14q + 30q^{2} + 45q^{3} + 84q^{4} + 129q^{5} + 214q^{6} + \ldots  &  q^{-3} + q^{-2} + 3q^{-1} + 9q + 13q^{2} + 23q^{3} + 33q^{4} + 52q^{5} + 75q^{6} + \ldots  &  240 \\
73  &&  q^{-4} + 3q^{-2} + 3q^{-1} + 12q + 25q^{2} + 36q^{3} + 66q^{4} + 97q^{5} + 156q^{6} + \ldots  &  q^{-3} + q^{-2} + 3q^{-1} + 8q + 11q^{2} + 19q^{3} + 26q^{4} + 40q^{5} + 57q^{6} + \ldots  &  240 \\
85  &&  q^{-4} + 2q^{-2} + q^{-1} + 8q + 15q^{2} + 20q^{3} + 34q^{4} + 51q^{5} + 85q^{6} + \ldots  &  q^{-3} + q^{-2} + 3q^{-1} + 6q + 8q^{2} + 14q^{3} + 19q^{4} + 26q^{5} + 35q^{6} + \ldots  &  508 \\
93  &&  q^{-4} + 3q^{-2} + 3q^{-1} + 9q + 16q^{2} + 18q^{3} + 36q^{4} + 46q^{5} + 66q^{6} + \ldots  &  q^{-3} + q^{-1} + 3q + 3q^{2} + 9q^{3} + 9q^{4} + 13q^{5} + 21q^{6} + 27q^{7} + \ldots  &  240 \\
103  &&  q^{-4} + 3q^{-2} + 3q^{-1} + 7q + 13q^{2} + 16q^{3} + 27q^{4} + 34q^{5} + 51q^{6} + \ldots  &  q^{-3} + q^{-1} + 3q + 3q^{2} + 6q^{3} + 7q^{4} + 11q^{5} + 14q^{6} + 20q^{7} + \ldots  &  240 \\
106  &&  q^{-4} + 2q^{-2} + q^{-1} + 4q + 11q^{2} + 10q^{3} + 24q^{4} + 23q^{5} + 45q^{6} + \ldots  &  q^{-3} + q^{-2} + 2q^{-1} + 4q + 5q^{2} + 9q^{3} + 9q^{4} + 15q^{5} + 17q^{6} + \ldots  &  240 \\
107  &&  q^{-4} + q^{-2} + q^{-1} + 4q + 8q^{2} + 10q^{3} + 17q^{4} + 23q^{5} + 35q^{6} + 47q^{7} + \ldots  &  q^{-3} + q^{-2} + 2q^{-1} + 4q + 5q^{2} + 8q^{3} + 10q^{4} + 14q^{5} + 18q^{6} + \ldots  &  240 \\
115  &&  q^{-4} + 3q^{-2} + 3q^{-1} + 5q + 10q^{2} + 14q^{3} + 22q^{4} + 25q^{5} + 36q^{6} + \ldots  &  q^{-3} + q^{-1} + 3q + 3q^{2} + 4q^{3} + 5q^{4} + 8q^{5} + 11q^{6} + 15q^{7} + \ldots  &  508 \\
122  &&  q^{-4} + 2q^{-2} + q^{-1} + 3q + 9q^{2} + 7q^{3} + 17q^{4} + 16q^{5} + 30q^{6} + 30q^{7} + \ldots  &  q^{-3} + 2q^{-1} + 3q + 2q^{2} + 6q^{3} + 4q^{4} + 9q^{5} + 8q^{6} + 15q^{7} + \ldots  &  366 \\
129  &&  q^{-4} + 2q^{-2} + 2q^{-1} + 5q + 7q^{2} + 7q^{3} + 15q^{4} + 16q^{5} + 22q^{6} + \ldots  &  q^{-3} + q^{-2} + q^{-1} + 2q + 4q^{2} + 6q^{3} + 5q^{4} + 8q^{5} + 11q^{6} + \ldots  &  240 \\
133  &&  q^{-4} + 3q^{-2} + 3q^{-1} + 5q + 8q^{2} + 9q^{3} + 15q^{4} + 19q^{5} + 26q^{6} + \ldots  &  q^{-3} - q^{-2} + q + 3q^{3} + 2q^{4} + 2q^{5} + 3q^{6} + 6q^{7} + 7q^{8} + 11q^{9} + \ldots  &  764 \\
134  &&  q^{-4} + q^{-2} + q^{-1} + 2q + 6q^{2} + 5q^{3} + 12q^{4} + 11q^{5} + 20q^{6} + 20q^{7}+ \ldots  &  q^{-3} + q^{-2} + q^{-1} + 3q + 3q^{2} + 5q^{3} + 5q^{4} + 8q^{5} + 9q^{6} + 13q^{7} + \ldots  &  240 \\
146  &&  q^{-4} + q^{-2} + q^{-1} + 2q + 5q^{2} + 4q^{3} + 10q^{4} + 9q^{5} + 16q^{6} + 15q^{7} + \ldots  &  q^{-3} + q^{-2} + q^{-1} + 2q + 3q^{2} + 5q^{3} + 4q^{4} + 6q^{5} + 7q^{6} + 11q^{7} + \ldots  &  240 \\
154  &&  q^{-4} + q^{-2} + q + 4q^{2} + 3q^{3} + 10q^{4} + 7q^{5} + 13q^{6} + 13q^{7} + 23q^{8} + \ldots  &  q^{-3} + q^{-2} + q^{-1} + 2q + 3q^{2} + 5q^{3} + 3q^{4} + 5q^{5} + 6q^{6} + 9q^{7} + \ldots  &  881 \\
158  &&  q^{-4} + 2q^{-2} + q^{-1} + 2q + 6q^{2} + 4q^{3} + 10q^{4} + 8q^{5} + 15q^{6} + 14q^{7} + \ldots  &  q^{-3} + q^{-1} + 2q + q^{2} + 3q^{3} + 2q^{4} + 5q^{5} + 4q^{6} + 7q^{7} + 6q^{8} + \ldots  &  240 \\
161  &&  q^{-4} + q^{-2} + q^{-1} + 3q + 4q^{2} + 4q^{3} + 7q^{4} + 7q^{5} + 10q^{6} + 13q^{7} + \ldots  &  q^{-3} + q^{-1} + q + q^{2} + 3q^{3} + 3q^{4} + 4q^{5} + 5q^{6} + 6q^{7} + 6q^{8} + \ldots  &  698 \\
165  &&  q^{-4} + 2q^{-2} + q^{-1} + 2q + 5q^{2} + 4q^{3} + 10q^{4} + 8q^{5} + 10q^{6} + 15q^{7} + \ldots  &  q^{-3} + q^{-1} + 2q + q^{2} + 3q^{3} + 2q^{4} + 3q^{5} + 5q^{6} + 6q^{7} + 5q^{8} + \ldots  &  908 \\
166  &&  q^{-4} + 2q^{-2} + q^{-1} + 2q + 5q^{2} + 4q^{3} + 9q^{4} + 7q^{5} + 14q^{6} + 12q^{7} + \ldots  &  q^{-3} + q^{-1} + 2q + q^{2} + 3q^{3} + 2q^{4} + 4q^{5} + 3q^{6} + 6q^{7} + 5q^{8} + \ldots  &  240 \\
167  &&   q^{-4} + q^{-2} + q^{-1} + 2q + 4q^{2} + 4q^{3} + 6q^{4} + 7q^{5} + 10q^{6} + 12q^{7} + \ldots  &  q^{-3} + q^{-2} + q^{-1} + 2q + 2q^{2} + 3q^{3} + 4q^{4} + 5q^{5} + 6q^{6} + 7q^{7} + \ldots  &  240 \\
170  &&  q^{-4} + q^{-1} + 2q + 3q^{2} + 2q^{3} + 6q^{4} + 5q^{5} + 9q^{6} + 8q^{7} + 14q^{8} + \ldots  &  q^{-3} + q^{-2} + q^{-1} + 2q + 2q^{2} + 4q^{3} + 3q^{4} + 4q^{5} + 5q^{6} + 8q^{7} + \ldots  &  508 \\
177  &&  q^{-4} + q^{-2} + q^{-1} + 2q + 4q^{2} + 3q^{3} + 5q^{4} + 7q^{5} + 8q^{6} + 10q^{7} + \ldots  &  q^{-3} + q^{-1} + q + q^{2} + 3q^{3} + 2q^{4} + 3q^{5} + 4q^{6} + 4q^{7} + 5q^{8} + \ldots  &  240 \\
186  &&  q^{-4} + q^{-2} + q^{-1} + q + 4q^{2} + 2q^{3} + 6q^{4} + 6q^{5} + 8q^{6} + 7q^{7} + \ldots  &  q^{-3} + q^{-1} + q + q^{2} + 3q^{3} + q^{4} + 3q^{5} + 3q^{6} + 5q^{7} + 4q^{8} + \ldots  &  240 \\
191  &&  q^{-4} + q^{-2} + q^{-1} + 2q + 3q^{2} + 3q^{3} + 5q^{4} + 5q^{5} + 7q^{6} + 8q^{7} + \ldots  &  q^{-3} + q^{-1} + q + q^{2} + 2q^{3} + 2q^{4} + 3q^{5} + 3q^{6} + 4q^{7} + 4q^{8} + \ldots  &  240 \\
205  &&  q^{-4} + q^{-2} + q^{-1} + 2q + 2q^{2} + 3q^{3} + 5q^{4} + 4q^{5} + 6q^{6} + 6q^{7} + \ldots  &  q^{-3} + q^{-1} + q + q^{2} + 2q^{3} + 2q^{4} + 2q^{5} + 2q^{6} + 4q^{7} + 3q^{8} + \ldots  &  508 \\
206  &&  q^{-4} + q^{-2} + q^{-1} + q + 3q^{2} + 2q^{3} + 5q^{4} + 4q^{5} + 7q^{6} + 6q^{7} + \ldots  &  q^{-3} + q^{-1} + q + q^{2} + 2q^{3} + q^{4} + 3q^{5} + 2q^{6} + 4q^{7} + 3q^{8} + \ldots  &  240 \\
209  &&  q^{-4} + q + 2q^{2} + q^{3} + 3q^{4} + 3q^{5} + 5q^{6} + 6q^{7} + 8q^{8} + 8q^{9} + 10q^{10} + \ldots  &  q^{-3} + q^{-2} + q^{-1} + q + 2q^{2} + 3q^{3} + 2q^{4} + 3q^{5} + 3q^{6} + 4q^{7} + \ldots  &  380 \\
213  &&  q^{-4} + q^{-2} + q^{-1} + q + 3q^{2} + 2q^{3} + 4q^{4} + 5q^{5} + 5q^{6} + 6q^{7} + \ldots  &  q^{-3} + q^{-1} + q + q^{2} + 2q^{3} + q^{4} + 2q^{5} + 3q^{6} + 3q^{7} + 3q^{8} + \ldots  &  240 \\
215  &&  q^{-4} + 2q^{-2} + 2q^{-1} + q + 3q^{2} + 4q^{3} + 5q^{4} + 5q^{5} + 7q^{6} + 8q^{7} + \ldots  &  q^{-3} - q^{-2} + q + q^{2} + q^{5} + q^{6} + 2q^{7} + q^{8} + 3q^{9} + 2q^{10} + \ldots  &  508 \\
221  &&  q^{-4} + q^{-2} + q^{-1} + q + 3q^{2} + 2q^{3} + 4q^{4} + 4q^{5} + 5q^{6} + 5q^{7} + 8q^{8} + \ldots  &  q^{-3} - q^{-2} + q^{-1} + q + 2q^{3} + q^{5} + q^{6} + 2q^{7} + 2q^{8} + 3q^{9} + \ldots  &  359 \\
230  &&  q^{-4} + q^{-2} + q^{-1} + q + 2q^{2} + 2q^{3} + 4q^{4} + 3q^{5} + 6q^{6} + 4q^{7} + 9q^{8} + \ldots  &  q^{-3} + q^{-1} + q + q^{2} + 2q^{3} + q^{4} + 2q^{5} + q^{6} + 3q^{7} + 2q^{8} + \ldots  &  508 \\
255  &&  q^{-4} + 2q^{-2} + q^{-1} + 2q + 3q^{2} + 2q^{3} + 4q^{4} + 3q^{5} + 4q^{6} + 4q^{7} + \ldots  &  q^{-3} + q^{-2} + 2q^{2} + 2q^{3} + q^{4} + 2q^{5} + 2q^{6} + 2q^{7} + 3q^{8} + 4q^{9} + \ldots  &  508 \\
266  &&  q^{-4} + q^{-2} + q^{-1} + q + 2q^{2} + q^{3} + 3q^{4} + 3q^{5} + 4q^{6} + 3q^{7} + 6q^{8} + \ldots  &  q^{-3} - q^{-2} + 2q^{-1} + q + q^{3} + 2q^{5} + q^{6} + 2q^{7} + q^{8} + 3q^{9} + \ldots  &  764 \\
285  &&  q^{-4} + q^{-2} + q + 2q^{2} + q^{3} + 2q^{4} + 2q^{5} + 3q^{6} + 3q^{7} + 5q^{8} + 3q^{9} + \ldots  &  q^{-3} + q^{-1} + q + 2q^{3} + q^{4} + q^{5} + q^{6} + q^{7} + 2q^{8} + 3q^{9} + 2q^{10} + \ldots  &  508 \\
286  &&  q^{-4} + q^{-2} + q + 2q^{2} + q^{3} + 3q^{4} + q^{5} + 3q^{6} + 2q^{7} + 5q^{8} + 4q^{9} + \ldots  &  q^{-3} + q^{-1} + q + q^{3} + q^{4} + 2q^{5} + q^{6} + 2q^{7} + q^{8} + 3q^{9} + 2q^{10} + \ldots  &  359 \\
287  &&  q^{-4} + q^{-2} + q^{-1} + q + 2q^{2} + 2q^{3} + 2q^{4} + 2q^{5} + 3q^{6} + 3q^{7} + 4q^{8} + \ldots  &  q^{-3} + q^{-2} + q^{-1} + q + q^{2} + q^{3} + 2q^{4} + 2q^{5} + 2q^{6} + 2q^{7} + \ldots  &  577 \\
299  &&  q^{-4} - q^{-2} + q^{-1} + q + q^{3} + 2q^{4} + 2q^{5} + q^{6} + q^{7} + 2q^{8} + 2q^{9} + \ldots  &  q^{-3} + q^{-2} + q + q^{2} + q^{3} + q^{4} + q^{5} + 2q^{6} + 2q^{7} + 2q^{8} + 3q^{9} + \ldots  &  359 \\
330  &&  q^{-4} + q^{-1} + q^{2} + 2q^{4} + 2q^{5} + 2q^{6} + q^{7} + 4q^{8} + 3q^{9} + 3q^{10} + \ldots  &  q^{-3} + q^{-1} + q^{2} + q^{3} + q^{5} + q^{6} + 2q^{7} + q^{8} + 3q^{9} + q^{10} + \ldots  &  908 \\
357  &&  q^{-4} + q + q^{2} + q^{3} + q^{4} + q^{5} + q^{6} + q^{7} + 3q^{8} + q^{9} + 2q^{10} + 3q^{11} + \ldots  &  q^{-3} + q^{-1} + q^{3} + q^{4} + q^{5} + q^{6} + q^{7} + q^{8} + 2q^{9} + q^{10} + 2q^{11} + \ldots  &  718 \\
390  &&  q^{-4} + 2q^{-2} + q^{-1} + q^{2} + q^{3} + 3q^{4} + q^{5} + 2q^{6} + 2q^{7} + 4q^{8} + 2q^{9} + \ldots  &  q^{-3} - q^{-2} + q + q^{3} + q^{7} - q^{8} + q^{9} + q^{12} + q^{13} + 2q^{15} + q^{16} + \ldots  &  838 \\
\end{array}
$$
\end{sidewaystable}

\begin{table}
\caption{\label{pol:genus two}Polynomial relation satisfied by $x=j_{1;N}$ and $y=j_{2;N}$
for the genus two moonshine groups $\Gamma_0(N)^+$.}
$$
\begin{array}{rc}
N & \text{polynomial relation} \\[1ex]
67 & y^3 - x^4 + 4y^2x + 5yx^2 - 15x^3 + 47y^2 + 121yx - 34x^2 + 724y + 602x + 3348 = 0 \\
73 & y^3 - x^4 + 4y^2x + 5yx^2 - 15x^3 + 43y^2 + 111yx - 39x^2 + 606y + 467x + 2508 = 0 \\
85 & y^3 - x^4 + 4y^2x + 8yx^2 - 7x^3 + 28y^2 + 88yx + 31x^2 + 277y + 367x + 864 = 0 \\
93 & y^3 - x^4 - 5yx^2 - 9x^3 + 16y^2 - 27yx - 57x^2 + 49y - 189x - 66 = 0 \\
103 & y^3 - x^4 - 5yx^2 - 9x^3 + 16y^2 - 21yx - 60x^2 + 65y - 164x + 18 = 0 \\
106 & y^3 - x^4 + 4y^2x + 4yx^2 - 7x^3 + 22y^2 + 55yx + 4x^2 + 158y + 160x + 360 = 0 \\
107 & y^3 - x^4 + 4y^2x + 7yx^2 - 5x^3 + 19y^2 + 58yx + 20x^2 + 128y + 164x + 272 = 0 \\
115 & y^3 - x^4 - 5yx^2 - 9x^3 + 16y^2 - 15yx - 59x^2 + 79y - 105x + 120 = 0 \\
122 & y^3 - x^4 + 2yx^2 - 3x^3 + 12y^2 + 3yx + 6x^2 + 48y + 12x + 64 = 0 \\
129 & y^3 - x^4 + 4y^2x - 10x^3 + 13y^2 + 21yx - 34x^2 + 48y - 18x + 36 = 0 \\
133 & y^3 - x^4 - 4y^2x - 7yx^2 - 3x^3 - 3y^2 - 77yx - 47x^2 - 84y - 243x - 198 = 0 \\
134 & y^3 - x^4 + 4y^2x + 3yx^2 - 5x^3 + 15y^2 + 35yx + 2x^2 + 72y + 66x + 108 = 0 \\
146 & y^3 - x^4 + 4y^2x + 3yx^2 - 5x^3 + 11y^2 + 27yx + 3x^2 + 40y + 43x + 48 = 0 \\
154 & y^3 - x^4 + 4y^2x + 3yx^2 - 2x^3 + 10y^2 + 29yx + 18x^2 + 33y + 49x + 36 = 0 \\
158 & y^3 - x^4 - 2yx^2 - 3x^3 + 10y^2 - 3yx - 11x^2 + 31y - 13x + 28 = 0 \\
161 & y^3 - x^4 + yx^2 - 3x^3 + 4y^2 - 3yx - x^2 = 0 \\
165 & y^3 - x^4 - 2yx^2 - 3x^3 + 10y^2 - 3yx - 8x^2 + 27y - 3x + 18 = 0 \\
166 & y^3 - x^4 - 2yx^2 - 3x^3 + 10y^2 - 3yx - 8x^2 + 34y - 8x + 40 = 0 \\
167 & y^3 - x^4 + 4y^2x + 3yx^2 - 5x^3 + 11y^2 + 23yx - 6x^2 + 35y + 15x + 25 = 0 \\
170 & y^3 - x^4 + 4y^2x + 6yx^2 - 3x^3 + 10y^2 + 26yx + 7x^2 + 25y + 21x = 0 \\
177 & y^3 - x^4 + yx^2 - 3x^3 + 4y^2 - x^2 + 5y + 2 = 0 \\
186 & y^3 - x^4 + yx^2 - 3x^3 + 4y^2 + 3yx - x^2 + 5y + 3x + 2 = 0 \\
191 & y^3 - x^4 + yx^2 - 3x^3 + 4y^2 - 2x^2 + 4y - x + 1 = 0 \\
205 & y^3 - x^4 + yx^2 - 3x^3 + 4y^2 + x^2 + y + 4x - 2 = 0 \\
206 & y^3 - x^4 + yx^2 - 3x^3 + 4y^2 + 3yx - 2x^2 + 7y + 2x + 4 = 0 \\
209 & y^3 - x^4 + 4y^2x + 6yx^2 + 5y^2 + 17yx + 13x^2 + 11y + 22x + 10 = 0 \\
213 & y^3 - x^4 + yx^2 - 3x^3 + 4y^2 + 3yx - 2x^2 + 6y + 3x + 3 = 0 \\
215 & y^3 - x^4 - 4y^2x - 4yx^2 - 2x^3 - y^2 - 33yx - 32x^2 + y - 52x - 21 = 0 \\
221 & y^3 - x^4 - 4y^2x + 3yx^2 - x^3 + 5y^2 - 14yx + 6x^2 + 8y - 8x + 4 = 0 \\
230 & y^3 - x^4 + yx^2 - 3x^3 + 4y^2 + 3yx + x^2 + 7y + 7x + 4 = 0 \\
255 & y^3 - x^4 + 4y^2x - 4yx^2 - 7x^3 - 2y^2 + 7yx - 20x^2 - 17y + 7x + 18 = 0 \\
266 & y^3 - x^4 - 4y^2x + 7yx^2 - x^3 + 5y^2 - 17yx + 11x^2 + 10y - 15x + 6 = 0 \\
285 & y^3 - x^4 + yx^2 + 4y^2 - 3yx + 4x^2 + 6y - 6x + 3 = 0 \\
286 & y^3 - x^4 + yx^2 + 4y^2 - 3yx + 5y - 3x + 2 = 0 \\
287 & y^3 - x^4 + 4y^2x + 3yx^2 - 5x^3 + 7y^2 + 14yx - 9x^2 + 11y - 7x - 6 = 0 \\
299 & y^3 - x^4 + 4y^2x + 5yx^2 - x^3 + 5y^2 + 11yx + 3x^2 + 2y - 6x - 12 = 0 \\
330 & y^3 - x^4 + 4yx^2 - 3x^3 - 2y^2 + 9yx - 2x^2 + 3y + 3x + 6 = 0 \\
357 & y^3 - x^4 + 4yx^2 - 2y^2 - 3yx + x^2 - 3x + 3 = 0 \\
390 & y^3 - x^4 - 4y^2x - 4yx^2 + x^3 - 27yx - 16x^2 + 18y + 56 = 0 \\
\end{array}
$$
\end{table}

\begin{sidewaystable}
\caption{\label{tab:y & x genus three}$q$-expansion of $y = j_{2;N}(z)$, and $q$-expansion
of $x = j_{1;N}(z)$, and $\kappa$ for the genus three moonshine groups $\Gamma_0(N)^+$.
For $N=510$ the value of $\kappa$ is for the generators $v=y+x$ and $w=y-x$.}
$$
\begin{array}{ccllc}
N   &&\hspace*{.18\textheight} y  &\hspace*{.18\textheight}  x  &  \kappa \\[1ex]
97  &&  q^{-5} + 2q^{-3} + q^{-2} + 5q^{-1} + 14q + 20q^{2} + 38q^{3} + 54q^{4} + 92q^{5} + \ldots  &  q^{-4} + q^{-3} + 3q^{-2} + 4q^{-1} + 11q + 18q^{2} + 25q^{3} + 39q^{4} + \ldots  &  487 \\
109  &&  q^{-5} + 2q^{-4} + 3q^{-2} + 4q^{-1} + 14q + 25q^{2} + 35q^{3} + 61q^{4} + 89q^{5} + \ldots  &  q^{-3} + q^{-2} + 2q^{-1} + 4q + 5q^{2} + 8q^{3} + 9q^{4} + 13q^{5} + 17q^{6} + \ldots  &  508 \\
113  &&  q^{-5} + 2q^{-3} + 2q^{-2} + 5q^{-1} + 13q + 16q^{2} + 28q^{3} + 38q^{4} + 60q^{5} + \ldots  &  q^{-4} + q^{-3} + 2q^{-2} + 3q^{-1} + 7q + 12q^{2} + 16q^{3} + 24q^{4} + 32q^{5} + \ldots  &  487 \\
127  &&  q^{-5} + q^{-3} + q^{-2} + 4q^{-1} + 7q + 10q^{2} + 17q^{3} + 22q^{4} + 35q^{5} + \ldots  &  q^{-4} + q^{-3} + 2q^{-2} + 2q^{-1} + 6q + 9q^{2} + 12q^{3} + 18q^{4} + 23q^{5} + \ldots  &  487 \\
139  &&  q^{-5} + q^{-3} + q^{-2} + 4q^{-1} + 7q + 8q^{2} + 14q^{3} + 17q^{4} + 28q^{5} + 35q^{6} + \ldots  &  q^{-4} + q^{-3} + 2q^{-2} + 2q^{-1} + 5q + 8q^{2} + 10q^{3} + 15q^{4} + 18q^{5} + \ldots  &  487 \\
149  &&  q^{-5} + q^{-3} + 2q^{-2} + 3q^{-1} + 7q + 8q^{2} + 12q^{3} + 16q^{4} + 24q^{5} + \ldots  &  q^{-4} + q^{-3} + q^{-2} + 2q^{-1} + 4q + 6q^{2} + 8q^{3} + 11q^{4} + 14q^{5} + \ldots  &  487 \\
151  &&  q^{-5} + q^{-4} + q^{-2} + 3q^{-1} + 6q + 8q^{2} + 12q^{3} + 18q^{4} + 25q^{5} + 33q^{6} + \ldots  &  q^{-3} + q^{-2} + q^{-1} + 2q + 3q^{2} + 4q^{3} + 4q^{4} + 6q^{5} + 7q^{6} + 9q^{7} + \ldots  &  508 \\
178  &&  q^{-5} + 2q^{-3} + 3q^{-1} + 6q + 4q^{2} + 11q^{3} + 8q^{4} + 19q^{5} + 16q^{6} + \ldots  &  q^{-4} + 2q^{-2} + q^{-1} + 2q + 5q^{2} + 3q^{3} + 8q^{4} + 6q^{5} + 11q^{6} + \ldots  &  487 \\
179  &&  q^{-5} + q^{-4} + q^{-2} + 2q^{-1} + 4q + 6q^{2} + 8q^{3} + 12q^{4} + 16q^{5} + 20q^{6} + \ldots  &  q^{-3} + q^{-2} + q^{-1} + 2q + 2q^{2} + 3q^{3} + 3q^{4} + 4q^{5} + 5q^{6} + 6q^{7} + \ldots  &  508 \\
183  &&   q^{-5} + q^{-3} + 2q^{-1} + 3q + 4q^{2} + 7q^{3} + 7q^{4} + 14q^{5} + 14q^{6} + 19q^{7} + \ldots  &  q^{-4} + q^{-3} + 2q^{-2} + 2q^{-1} + 4q + 5q^{2} + 6q^{3} + 9q^{4} + 9q^{5} + \ldots  &  635 \\
185  &&  q^{-5} + q^{-3} + q^{-2} + 2q^{-1} + 4q + 4q^{2} + 7q^{3} + 9q^{4} + 14q^{5} + 15q^{6} + \ldots  &  q^{-4} + q^{-2} + 2q + 3q^{2} + 3q^{3} + 4q^{4} + 5q^{5} + 8q^{6} + 9q^{7} + 13q^{8} + \ldots  &  487 \\
187  &&   q^{-5} + q^{-2} + 3q^{-1} + 3q + 5q^{2} + 6q^{3} + 6q^{4} + 10q^{5} + 12q^{6} + 18q^{7} + \ldots  &  q^{-4} + q^{-3} + q^{-2} + q^{-1} + 3q + 3q^{2} + 5q^{3} + 7q^{4} + 8q^{5} + \ldots  &  515 \\
194  &&  q^{-5} + 2q^{-3} + q^{-2} + 3q^{-1} + 6q + 4q^{2} + 10q^{3} + 8q^{4} + 16q^{5} + 14q^{6} + \ldots  &  q^{-4} + q^{-3} + q^{-2} + 2q^{-1} + 3q + 4q^{2} + 5q^{3} + 7q^{4} + 7q^{5} + \ldots  &  487 \\
203  &&  q^{-5} + q^{-3} + q^{-2} + 3q^{-1} + 4q + 4q^{2} + 7q^{3} + 7q^{4} + 10q^{5} + 13q^{6} + \ldots  &  q^{-4} + q^{-3} + q^{-2} + q^{-1} + 2q + 4q^{2} + 4q^{3} + 6q^{4} + 7q^{5} + 8q^{6} + \ldots  &  1015 \\
217  &&  q^{-5} + 2q^{-3} + q^{-2} + 3q^{-1} + 4q + 5q^{2} + 7q^{3} + 8q^{4} + 11q^{5} + 12q^{6} + \ldots  &  q^{-4} - q^{-3} + q^{-2} - q^{-1} + q^{2} + 2q^{4} + 3q^{5} + 3q^{6} + 3q^{7} + 5q^{8} + \ldots  &  905 \\
239  &&  q^{-5} + q^{-3} + q^{-2} + 2q^{-1} + 3q + 3q^{2} + 5q^{3} + 5q^{4} + 7q^{5} + 8q^{6} + \ldots  &  q^{-4} + q^{-3} + q^{-2} + q^{-1} + 2q + 3q^{2} + 3q^{3} + 4q^{4} + 5q^{5} + 6q^{6} + \ldots  &  487 \\
246  &&  q^{-5} + 2q^{-1} + 2q + 2q^{2} + 3q^{3} + 2q^{4} + 6q^{5} + 4q^{6} + 9q^{7} + 8q^{8} + \ldots  &  q^{-4} + q^{-2} + q + 2q^{2} + q^{3} + 4q^{4} + 2q^{5} + 4q^{6} + 4q^{7} + 7q^{8} + \ldots  &  635 \\
249  &&  q^{-5} + q^{-3} + q^{-2} + q^{-1} + 3q + 2q^{2} + 4q^{3} + 5q^{4} + 6q^{5} + 7q^{6} + \ldots  &  q^{-4} + q^{-1} + q + 2q^{2} + q^{3} + 2q^{4} + 3q^{5} + 3q^{6} + 3q^{7} + 6q^{8} + \ldots  &  635 \\
258  &&  q^{-5} + 2q^{-3} + q^{-2} + 2q^{-1} + 4q + 2q^{2} + 6q^{3} + 5q^{4} + 8q^{5} + 7q^{6} + \ldots  &  q^{-4} + q^{-3} + q^{-2} + q^{-1} + q + 3q^{2} + 3q^{3} + 4q^{4} + 4q^{5} + 5q^{6} + \ldots  &  635 \\
282  &&  q^{-4} + q^{-2} + q + 2q^{2} + q^{3} + 3q^{4} + q^{5} + 3q^{6} + 3q^{7} + 5q^{8} + 3q^{9} + \ldots  &  q^{-3} + q^{-1} + q + 2q^{3} + q^{4} + q^{5} + q^{6} + 2q^{7} + q^{8} + 3q^{9} + 2q^{10} + \ldots  &  282 \\
290  &&  q^{-5} + q^{-3} + q^{-1} + 2q + q^{2} + 3q^{3} + 2q^{4} + 6q^{5} + 4q^{6} + 7q^{7} + 5q^{8} + \ldots  &  q^{-4} + 2q^{-2} + q^{-1} + q + 3q^{2} + q^{3} + 4q^{4} + 2q^{5} + 3q^{6} + 3q^{7} + \ldots  &  487 \\
295  &&  q^{-5} + q^{-1} + q + q^{2} + 2q^{3} + 2q^{4} + 4q^{5} + 3q^{6} + 4q^{7} + 5q^{8} + 6q^{9} + \ldots  &  q^{-4} + q^{-3} + q^{-2} + q^{-1} + 2q + 2q^{2} + 2q^{3} + 3q^{4} + 3q^{5} + 4q^{6} + \ldots  &  487 \\
303  &&  q^{-5} + q^{-2} + q^{-1} + 2q + q^{2} + 2q^{3} + 3q^{4} + 3q^{5} + 3q^{6} + 6q^{7} + 5q^{8} + \ldots  &  q^{-4} + q^{-3} + q^{-1} + q + 2q^{2} + 2q^{3} + 2q^{4} + 3q^{5} + 3q^{6} + 3q^{7} + \ldots  &  635 \\
310  &&  q^{-5} + q^{-1} + q + q^{2} + 2q^{3} + q^{4} + 4q^{5} + 2q^{6} + 5q^{7} + 4q^{8} + 6q^{9} + \ldots  &  q^{-4} + q^{-2} + q + 2q^{2} + q^{3} + 2q^{4} + q^{5} + 3q^{6} + 2q^{7} + 4q^{8} + \ldots  &  487 \\
318  &&  q^{-5} + q^{-3} - q^{-2} + q^{-1} + 2q + q^{2} + 3q^{3} + q^{4} + 4q^{5} + 2q^{6} + 5q^{7} + \ldots  &  q^{-4} + 2q^{-2} + q^{-1} + q + 2q^{2} + q^{3} + 3q^{4} + 2q^{5} + 3q^{6} + 3q^{7} + \ldots  &  635 \\
329  &&  q^{-5} + q^{-2} + q^{-1} + 2q + 2q^{2} + q^{3} + 2q^{4} + 3q^{5} + 3q^{6} + 4q^{7} + 4q^{8} + \ldots  &  q^{-4} + q^{-3} + q^{-1} + q + q^{2} + 2q^{3} + 2q^{4} + 2q^{5} + 3q^{6} + 3q^{7} + \ldots  &  905 \\
429  &&  q^{-5} + q^{-2} + q^{-1} + 2q + q^{2} + q^{3} + q^{4} + 2q^{5} + q^{6} + 3q^{7} + 2q^{8} + \ldots  &  q^{-4} + q^{-3} + q^{2} + q^{3} + q^{4} + q^{5} + 2q^{6} + 2q^{7} + 3q^{8} + 2q^{9} + \ldots  &  1187 \\
430  &&  q^{-5} + q^{-3} + q^{-1} + q + q^{2} + 2q^{3} + q^{4} + 3q^{5} + q^{6} + 3q^{7} + 2q^{8} + \ldots  &  q^{-4} - q^{-3} + q^{-2} + q^{4} + 2q^{6} + 2q^{8} + 2q^{10} + q^{11} + 2q^{12} + \ldots  &  487 \\
455  &&  q^{-5} - q^{-3} + q^{-2} + q + q^{2} + 2q^{5} + q^{6} + q^{7} + 2q^{8} + q^{9} + 2q^{10} + \ldots  &  q^{-4} - q^{-3} + q^{-2} + q + q^{2} + q^{4} + 2q^{8} + q^{9} + q^{10} + q^{11} + \ldots  &  1129 \\
462  &&  q^{-5} + q^{-3} + q^{-1} + q + q^{2} + 2q^{3} + q^{4} + 3q^{5} + q^{6} + 2q^{7} + 2q^{8} + \ldots  &  q^{-4} + q^{-2} + q + q^{2} + q^{4} + q^{5} + q^{6} + q^{7} + 2q^{8} + q^{9} + 3q^{10} + \ldots  &  1680 \\
510  &&  q^{-5} + \frac{1}{2}q^{-3} + \frac{1}{2}q^{-2} + q + q^{3} + \frac{3}{2}q^{4} + 2q^{5} + q^{6} + 2q^{7} + \frac{1}{2}q^{8} + \ldots  &  q^{-4} + \frac{1}{2}q^{-3} + \frac{1}{2}q^{-2} + q^{2} + q^{3} + \frac{3}{2}q^{4} + q^{5} + q^{6} + \frac{3}{2}q^{8} + q^{9} + \ldots  &  635 \\
\end{array}
$$
\end{sidewaystable}

\begin{sidewaystable}
\caption{\label{pol:genus three}Polynomial relation satisfied by $x=j_{1;N}$ and $y=j_{2;N}$
for the genus three moonshine groups $\Gamma_0(N)^+$.  For $N=510$, we give the
polynomial relation satisfied by $v=y+x$ and $w=y-x$, instead.}
$$
\begin{array}{rc}
N & \text{polynomial relation} \\[1ex]
97 & y^4 - x^5 + 5y^3x + 12y^2x^2 + 17yx^3 - 27x^4 + 64y^3 + 280y^2x + 545yx^2 - 128x^3 + 1685y^2 + 5965yx + 2558x^2 + 22370y + 31654x + 102984 = 0 \\
109 & y^3 - x^5 - yx^3 + 13y^2x - 17x^4 + 49yx^2 + 59y^2 - 32x^3 + 494yx + 743x^2 + 1146y + 4461x + 7254 = 0 \\
113 & y^4 - x^5 + 5y^3x + 7y^2x^2 + 3yx^3 - 26x^4 + 50y^3 + 162y^2x + 129yx^2 - 252x^3 + 870y^2 + 1629yx - 878x^2 + 6279y + 1397x + 12108 = 0 \\
127 & y^4 - x^5 + 5y^3x + 11y^2x^2 + 9yx^3 - 21x^4 + 39y^3 + 160y^2x + 209yx^2 - 139x^3 + 598y^2 + 1580yx - 63x^2 + 3968y + 2600x + 7344 = 0 \\
139 & y^4 - x^5 + 5y^3x + 11y^2x^2 + 9yx^3 - 21x^4 + 34y^3 + 140y^2x + 179yx^2 - 151x^3 + 455y^2 + 1176yx - 321x^2 + 2574y + 782x + 3084 = 0 \\
149 & y^4 - x^5 + 5y^3x + 6y^2x^2 - 18x^4 + 30y^3 + 87y^2x + 33yx^2 - 130x^3 + 289y^2 + 364yx - 426x^2 + 988y - 412x + 468 = 0 \\
151 & y^3 - x^5 + 2yx^3 + 4y^2x - 4x^4 + 12yx^2 + 19y^2 + 7x^3 + 55yx + 62x^2 + 118y + 176x + 240 = 0 \\
178 & y^4 - x^5 + 2y^2x^2 + 5yx^3 - 12x^4 + 6y^3 + 8y^2x + 44yx^2 - 44x^3 + 19y^2 + 117yx - 32x^2 + 106y + 97x + 132 = 0 \\
179 & y^3 - x^5 + 2yx^3 + 4y^2x - 4x^4 + 15yx^2 + 13y^2 + 8x^3 + 49yx + 69x^2 + 62y + 142x + 108 = 0 \\
183 & y^4 - x^5 + 5y^3x + 11y^2x^2 + 13yx^3 - 10x^4 + 24y^3 + 101y^2x + 171yx^2 + 6x^3 + 228y^2 + 739yx + 376x^2 + 1045y + 1435x + 1650 = 0 \\
185 & y^4 - x^5 + y^2x^2 - 4yx^3 - 8x^4 + 12y^3 + y^2x - 11yx^2 - 38x^3 + 52y^2 - 16yx - 84x^2 + 87y - 89x + 28 = 0 \\
187 & y^4 - x^5 + 5y^3x + 10y^2x^2 + 6yx^3 - 15x^4 + 20y^3 + 71y^2x + 60yx^2 - 95x^3 + 132y^2 + 203yx - 319x^2 + 234y - 570x - 432 = 0 \\
194 & y^4 - x^5 + 5y^3x + 2y^2x^2 - 3yx^3 - 11x^4 + 28y^3 + 44y^2x - 13yx^2 - 50x^3 + 165y^2 + 77yx - 112x^2 + 342y - 42x + 216 = 0 \\
203 & y^4 - x^5 + 5y^3x + 6y^2x^2 - yx^3 - 15x^4 + 16y^3 + 47y^2x + 9yx^2 - 90x^3 + 87y^2 + 82yx - 262x^2 + 144y - 359x - 173 = 0 \\
217 & y^4 - x^5 - 5y^3x + 2y^2x^2 - 5x^4 - 5y^3 - 38y^2x - yx^2 - 4x^3 - 97y^2 - 125yx - 22x^2 - 324y - 187x - 309 = 0 \\
239 & y^4 - x^5 + 5y^3x + 6y^2x^2 - yx^3 - 11x^4 + 15y^3 + 45y^2x + 13yx^2 - 52x^3 + 73y^2 + 86yx - 115x^2 + 114y - 107x - 26 = 0 \\
246 & y^4 - x^5 + 5y^2x^2 - 8x^4 + 5y^3 + 19y^2x + 7yx^2 - 22x^3 + 15y^2 + 14yx - 24x^2 - 5y - 9x = 0 \\
249 & y^4 - x^5 - 4y^2x^2 + yx^3 - 2x^4 + 8y^3 - 6y^2x - 12yx^2 + x^3 + 22y^2 - 20yx - 7x^2 + 24y - 14x + 9 = 0 \\
258 & y^4 - x^5 + 5y^3x + 2y^2x^2 - 8yx^3 - 7x^4 + 11y^3 + 20y^2x - 21yx^2 - 29x^3 + 36y^2 + 7yx - 56x^2 + 46y - 26x + 20 = 0 \\
282 & y^3 - x^4 + yx^2 + 4y^2 - 3yx + 4x^2 + 3y - 3x = 0 \\
290 & y^4 - x^5 + 6y^2x^2 + 5yx^3 - 6x^4 - 2y^3 + 24y^2x + 36yx^2 + 5x^3 + 15y^2 + 77yx + 72x^2 + 62y + 94x + 32 = 0 \\
295 & y^4 - x^5 + 5y^3x + 10y^2x^2 + 10yx^3 - 4x^4 + 12y^3 + 46y^2x + 65yx^2 + 9x^3 + 51y^2 + 136yx + 64x^2 + 90y + 99x + 47 = 0 \\
303 & y^4 - x^5 + 5y^3x + 5y^2x^2 + yx^3 - 7x^4 + 11y^3 + 23y^2x + 12yx^2 - 17x^3 + 27y^2 + 34yx - 11x^2 + 28y + 16x + 20 = 0 \\
310 & y^4 - x^5 + 5y^2x^2 - 4x^4 + 5y^3 + 17y^2x + 11yx^2 - 2x^3 + 23y^2 + 40yx + 10x^2 + 45y + 23x + 26 = 0 \\
318 & y^4 - x^5 + 6y^2x^2 + 9yx^3 - 6x^4 - 6y^3 + 10y^2x + 56yx^2 - 2x^3 + 9y^2 + 123yx + 26x^2 + 100y + 23x - 24 = 0 \\
329 & y^4 - x^5 + 5y^3x + 5y^2x^2 + yx^3 - 7x^4 + 11y^3 + 18y^2x - 2yx^2 - 20x^3 + 18y^2 - 17yx - 31x^2 - 18y - 28x - 12 = 0 \\
429 & y^4 - x^5 + 5y^3x + 5y^2x^2 - 4yx^3 - 7x^4 + y^3 - 4y^2x - 23yx^2 - 21x^3 - 3y^2 - 21yx - 22x^2 - 6y - 8x = 0 \\
430 & y^4 - x^5 - 5y^3x + 6y^2x^2 + 2yx^3 - 4x^4 - 5y^2x + 11yx^2 - 5x^3 - 2y^2 + 8yx - 4x^2 - 8y + 6x = 0 \\
455 & y^4 - x^5 - 5y^3x + 14y^2x^2 - 16yx^3 + 7x^4 + 3y^3 - 7y^2x + 13yx^2 + 2x^3 - 6y^2 + 30yx - 47x^2 - 27y + 66x - 27 = 0 \\
462 & y^4 - x^5 + y^2x^2 - 4x^4 + 5y^3 + 3y^2x - yx^2 - 8x^3 + 5y^2 + 4yx - 20x^2 - 11y - 19x - 24 = 0 \\
510 & w^5 \!-\! 5w^4v \!+\! 4w^4 \!+\! 10w^3v^2 \!-\! 6w^3v \!-\! 3w^3 \!-\! 10w^2v^3 \!-\! 4w^2v^2 \!+\! 46w^2v \!-\! 23w^2 \!+\! 5wv^4 \!+\! 34wv^3 \!-\! 75wv^2 \!+\! 97wv \!-\! 5w \!-\! v^5 \!+\! 4v^4 \!+\! 45v^3 \!-\! 78v^2 \!+\! 50v \!+\! 30 = 0 \\
\end{array}
$$
\end{sidewaystable}

To conclude the presentation of results, we describe some of the information we obtained for
genus two and genus three moonshine groups.  In all instances, we derived $q$-expansions for
the generators of the function fields out to $q^\kappa$, where $\kappa$ is determined according
to Remark~\ref{rem:kappa algorithm}.  With the exception of $j_{1,510}$ and $j_{2,510}$,
the $q$-expansions have integer coefficients.  The $q$-expansions of $j_{1,510}$ and
$j_{2,510}$ have half-integer coefficients.  With a further base change, namely
$v=j_{2,510}+j_{1,510}$ and $w=j_{2,510}-j_{1,510}$, the $q$-expansions for the generators
$v$ and $w$ have integer coefficients for $N=510$, also.  This together with
Theorem~\ref{genus_notzero_integrality} yields the following result and finally proves
Theorem~\ref{main theorem}.

\emph{For each moonshine group of genus up to three, the two generators $j_{1;N}$ and $j_{2;N}$
for $N\neq510$ and $j_{2;N}+j_{1;N}$ and $j_{2;N}-j_{1;N}$ for $N=510$ have integer $q$-expansions.
Moreover, the leading coefficients are one and the orders of the poles at $i\infty$ are at most
$g+2$.}

As it turned out, for all moonshine groups of genus two and for all but four moonshine groups of
genus three, the gap sequence consists of $\{q^{-1},\ldots,q^{-g}\}$, and $i\infty$ was not a
Weierstrass point.  For $\Gamma_0(109)^+$, $\Gamma_0(151)^+$ and $\Gamma_0(179)^+$, the gap
sequence consists of $\{q^{-1},q^{-2},q^{-4}\}$ and for $\Gamma_0(282)^+$ the gap sequence
consists of $\{q^{-1},q^{-2},q^{-5}\}$.  For the latter four groups, $i\infty$ is a Weierstrass
point.

We list in Tables~\ref{tab:y & x genus two} to~\ref{pol:genus three} the $q$-expansions of the
generators as well as the algebraic relation satisfied by the generators.

In \cite{JST13c} we will provide a thorough discussion of our findings in the setting of higher
genus moonshine groups, including expressions for the function field generators in terms of
holomorphic Eisenstein series, relations satisfied by holomorphic Eisenstein series, further
$q$-expansions, and the special role of level $N=510$.

\section{Concluding remarks}\label{concluding remark}

In a few words, we can express a response to the problem posed by T.~Gannon which we quoted in
the introduction.

\emph{The generators of the function fields associated to $X_{N}$
which one should consider are given by the two invariant
holomorphic functions whose poles at $i\infty$ have the smallest
possible orders and with initial $q$-expansions which are normalized
by the criteria of row-reduced echelon form.}

The above statement when applied to all the moonshine groups we considered produced generators
with integer $q$-expansions, except for the level $N=510$ where a further base change was
necessary.  The cases we considered include all moonshine groups $\Gamma_0(N)^+$ with squarefree
$N$ of genus up to three.

In other articles, we will present further results related to the material presented here.
Specifically, in \cite{JST13a}, we study further details regarding genus zero moonshine groups.
In particular, for each level, we list the generators of the ring of holomorphic forms as well
as relations for the Hauptmodul and Kronecker limit functions in terms of holomorphic Eisenstein
series.  Similar information for genus one moonshine groups is studied in \cite{JST13b}, as well
as additional aspects of the ``arithmetic of the moonshine elliptic curves''.

\section*{Acknowledgments}
The first named author (J.J.) discussed a preliminary version of this paper with Peter Sarnak
which included the findings from sections~\ref{parabolic Eisenstein} onward.  During the
discussion, Professor Sarnak suggested that we seek a proof of the integrality of all
coefficients in the $q$-expansion once we have shown that a certain number of the first
coefficients are integers.  His suggestion led us to formulate and prove the material in
section~\ref{integrality}.  We thank Professor Sarnak for generously sharing with us his
mathematical insight.

The classical material in section~\ref{additional aspects} came from a discussion between the
third named author (H.T.) and Abhishek Saha.  The approach using arithmetic algebraic geometry
stemmed from a discussion between the first named author (J.J.), Brian Conrad and an anonymous
reader of a previous draft of this article.  We thank these individuals for allowing us to
include their remarks and ideas.

Finally, the numerical computations for some levels were quite resource demanding.
We thank the Public Enterprise Electric Utility of Bosnia and Herzegovina for generously
granting us full access to one of their new 256 GB RAM computers, which enabled us to compute
for all moonshine groups of genus up to three.  We are very grateful for their support of our
academic research.

\end{document}